\documentclass[11pt]{article}

\usepackage{amsmath, amsfonts, amssymb, amsthm,  graphicx, mathtools, enumerate}

\usepackage[blocks, affil-it]{authblk}
\usepackage{bbm}
\usepackage[numbers, square]{natbib}
\usepackage[CJKbookmarks=true,
            bookmarksnumbered=true,
			bookmarksopen=true,
			colorlinks=true,
			citecolor=red,
			linkcolor=blue,
			anchorcolor=red,
			urlcolor=blue]{hyperref}
\usepackage[usenames]{color}

\usepackage[letterpaper, left=1.2truein, right=1.2truein, top = 1.2truein, bottom = 1.2truein]{geometry}
\usepackage[ruled, vlined, lined, commentsnumbered]{algorithm2e}

\usepackage{prettyref,soul}

\newtheorem{lemma}{Lemma}[section]
\newtheorem{proposition}{Proposition}[section]
\newtheorem{thm}{Theorem}[section]

\newrefformat{eq}{(\ref{#1})}
\newrefformat{chap}{Chapter~\ref{#1}}
\newrefformat{sec}{Section~\ref{#1}}
\newrefformat{algo}{Algorithm~\ref{#1}}
\newrefformat{fig}{Fig.~\ref{#1}}
\newrefformat{tab}{Table~\ref{#1}}
\newrefformat{rmk}{Remark~\ref{#1}}
\newrefformat{clm}{Claim~\ref{#1}}
\newrefformat{def}{Definition~\ref{#1}}
\newrefformat{cor}{Corollary~\ref{#1}}
\newrefformat{lmm}{Lemma~\ref{#1}}
\newrefformat{lemma}{Lemma~\ref{#1}}
\newrefformat{prop}{Proposition~\ref{#1}}
\newrefformat{app}{Appendix~\ref{#1}}
\newrefformat{ex}{Example~\ref{#1}}
\newrefformat{exer}{Exercise~\ref{#1}}
\newrefformat{soln}{Solution~\ref{#1}}
\newrefformat{cond}{Condition~\ref{#1}}

\def\Var{\textsf{Var}} 

\def\text#1{\mbox{\rm #1}}

\DeclarePairedDelimiter{\ceil}{\lceil}{\rceil}

\newcommand{\argmin}{\mathop{\rm argmin}}
\newcommand{\argmax}{\mathop{\rm argmax}}

\newcommand{\wh}{\widehat}

\newcommand{\fnorm}[1]{\|#1\|_{\rm F}}

\newcommand{\opnorm}[1]{\|#1\|_{\rm op}}

\newcommand{\rank}{\mathop{\sf rank}}
\newcommand{\Tr}{\mathop{\sf Tr}}

\newcommand{\iprod}[2]{\left \langle #1, #2 \right\rangle}

\newtheorem*{m2'}{Condition M2'}

\title{Phase Transitions in Approximate Ranking
}
\author{Chao Gao}
\affil{
University of Chicago

chaogao@galton.uchicago.edu
}

\begin{document}
\maketitle

\begin{abstract}
We study the problem of approximate ranking from observations of pairwise interactions. The goal is to estimate the underlying ranks of $n$ objects from data through interactions of comparison or collaboration. Under a general framework of approximate ranking models, we characterize the exact optimal statistical error rates of estimating the underlying ranks. We discover important phase transition boundaries of the optimal error rates. Depending on the value of the signal-to-noise ratio (SNR) parameter, the optimal rate, as a function of SNR, is either trivial, polynomial, exponential or zero. The four corresponding regimes thus have completely different error behaviors. To the best of our knowledge, this phenomenon, especially the phase transition between the polynomial and the exponential rates, has not been discovered before.
\smallskip

\textbf{Keywords}: minimax rate, permutation, sorting, pairwise comparison, latent space.
\end{abstract}


\section{Introduction}

Given data $\{X_{ij}\}_{1\leq i\neq j\leq n}$, we study recovery of the underlying ranks of the $n$ objects in the paper. The observation $X_{ij}$ can be interpreted as the outcome of an interaction between $i$ and $j$. For example, in sports, $X_{ij}$ can be the match result of a game between team $i$ and team $j$. In a coauthorship network, $X_{ij}$ can be the number of scientific papers jointly written by author $i$ and author $j$. We consider a very general approximate ranking model in the paper. It imposes the mean structure
$$\mathbb{E}X_{ij}=\mu_{r(i)r(j)}.$$
Here $r(i)\in\{1,2,...,n\}$ is the rank of object $i$. The interaction outcome $\mu_{r(i)r(j)}$ is only determined by the latent positions of $i$ and $j$, and $X_{ij}$ is thus a noisy measurement of $\mu_{r(i)r(j)}$. The goal of approximate ranking is to recover the underlying $r(i)$ for each $i\in\{1,2,...,n\}$.

In the literature, the problem of \textit{exact} ranking is a well studied topic, especially in the settings of pairwise comparison with Bernoulli outcomes. The goal of exact ranking assumes that the underlying $r$ is a permutation, and therefore estimating $r$ is equivalent to sorting the $n$ objects, which gives an alternative name ``noisy sorting" to such a problem \citep{braverman2008noisy}. We refer the readers to \cite{negahban2012iterative,shah2015simple,shah2016stochastically,mao2017minimax} and references therein for recent developments in this area.

In contrast, this paper studies the \textit{approximate} ranking problem. We do not impose the constraint that $r$ must be a permutation. More generally, we allow any $r$ that satisfies $r(i)\in\{1,2,...,n\}$ for each entry plus some moment conditions. This allows possible ties in the rank, and the ranks of the $n$ objects do not necessarily start from $1$ or end at $n$. The number $r(i)$ should instead be interpreted as a discrete latent position of object $i$. Such an approximate ranking setting is more natural for many applications. For example, it is conceivable in certain situations that there is a subset of objects that may behave very similarly through pairwise interactions. As a consequence, we can allow the same value $r(i)$ for all $i$ in the group in such a scenario. Moreover, the numbers $r(i)$'s in the approximate ranking setting not only reflect the order of the $n$ objects, but they also carry information about their relative differences through the interpretation of latent positions. These features and advantages distinguish the approximate ranking problem from the exact ranking problem studied in the literature.

The main contribution of the paper is the exact characterization of the optimal statistical error of the approximate ranking problem. Given an estimator $\hat{r}$, we measure the error through the loss function $\ell_2(\hat{r},r)=\frac{1}{n}\sum_{i=1}^n(\hat{r}(i)-r(i))^2$. With the signal parameter $\beta^2$ and the noise level $\sigma^2$ defined later in (\ref{eq:signal}) and (\ref{eq:error}), we show that the optimal rate of $\ell_2(\hat{r},r)$ is a function of the signal-to-noise ratio parameter $\text{SNR}=\frac{n\beta^2}{4\sigma^2}$. 
Our results are summarized in Figure \ref{fig:laoluan}.
\begin{figure}[h]\label{fig:laoluan}
\centering
\includegraphics[width=0.65\textwidth]{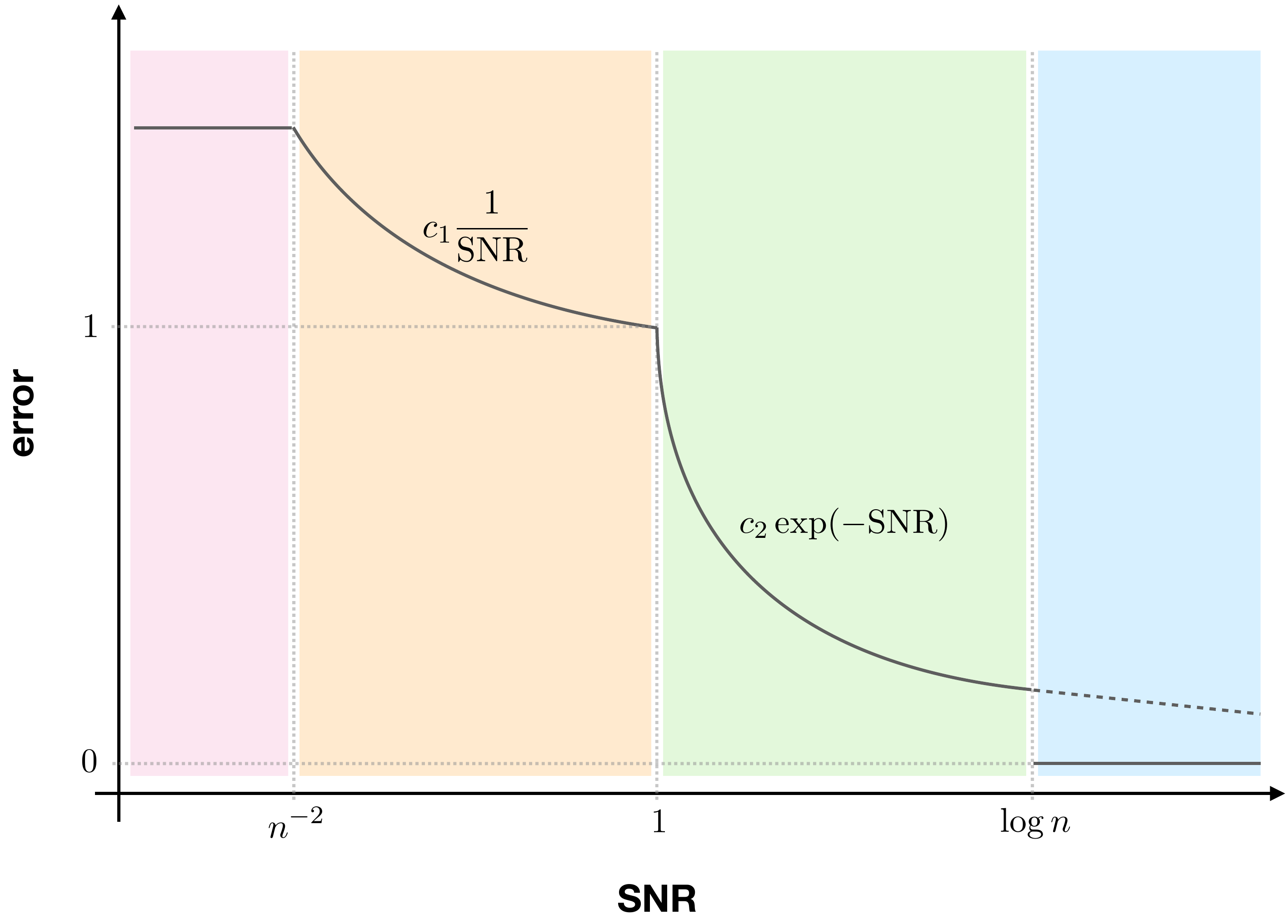}
\caption{Optimal error rate with respect to $\ell_2(\hat{r},r)$ as a function of SNR.}
\end{figure}
According to the plot in Figure \ref{fig:laoluan}, the optimal error exhibits interesting and delicate phase transition phenomena. Depending on the value of SNR, the optimal rate falls into four different regimes. In the first regime where the SNR is smaller than $n^{-2}$, the rate of $\ell_2(\hat{r},r)$ trivially takes the order of $\Theta(n^2)$, which is the largest possible value of $\ell_2(\hat{r},r)$. Next, as long as the SNR is greater than $n^{-2}$, the optimal rate starts to decrease polynomially fast. After the SNR passes the threshold of $1$, the optimal error accelerates to an exponential rate. In the final regime where $\text{SNR}>\log n$, we achieve exact recovery and $\ell_2(\hat{r},r)=0$ with high probability. The dashed curve in the final regime is the error in expectation, which still decreases at the rate of $\exp\left(-\text{SNR}\right)$. Besides the loss function $\ell_2(\hat{r},r)$, optimal rates and similar phase transition boundaries are also derived for the loss $\ell_1(\hat{r},r)=\frac{1}{n}\sum_{i=1}^n|\hat{r}(i)-r(i)|$.

The phase transition between the polynomial rate and the exponential rate is remarkable. To the best knowledge of the author, this is a new phenomenon first discovered in the approximate ranking problem. Mathematically speaking, the polynomial rate is driven by an entropy calculation argument, and when $\text{SNR}<1$, estimating $r$ is like estimating a continuous parameter in $\mathbb{R}^n$. In comparison, when $\text{SNR}>1$, the discrete nature of $r$ starts to come into effect, and we have sufficient information thanks to a high SNR to distinguish between $r(i)$ and its neighboring values for each $i$, which contributes to the exponential rate.

The paper is organized as follows. In Section \ref{sec:main}, we introduce the approximate ranking model, and then present the main results on optimal rates and phase transitions. In Section \ref{sec:alg}, we consider two special cases of the approximate ranking model, and derive optimal procedures and adaptive algorithms that can achieve the optimal rates. We then discuss a few related topics to our paper in Section \ref{sec:disc}. All proofs are given in Section \ref{sec:proof}.

We close this section by introducing notations that will be used later. For $a,b\in\mathbb{R}$, let $a\vee b=\max(a,b)$ and $a\wedge b=\min(a,b)$. For an integer $m$, $[m]$ denotes the set $\{1,2,...,m\}$. Given a set $S$, $|S|$ denotes its cardinality, and $\mathbb{I}_S$ is the associated indicator function. We use $\mathbb{P}$ and $\mathbb{E}$ to denote generic probability and expectation whose distribution is determined from the context. For two positive sequences $\{a_n\}$ and $\{b_n\}$, we use the notation $a_n\lesssim b_n$ if $a_n\leq Cb_n$ for all $n$ with some consntant $C>0$ independent of $n$. Finally, for two probability measures $\mathbb{P}$ and $\mathbb{Q}$, the Kullback-Leibler divergence is defined as $D(\mathbb{P}\|\mathbb{Q})=\int \log\frac{d\mathbb{P}}{d\mathbb{Q}}d\mathbb{P}$.

\section{Main Results}\label{sec:main}

\subsection{The Approximate Ranking Model}\label{sec:model}

Consider $n$ objects with ranks $r(1),r(2),...,r(n)\in[n]$.
We observe $\{X_{ij}\}_{1\leq i\neq j\leq n}$ that follow the generating process
\begin{equation}
X_{ij}=\mu_{r(i)r(j)}+Z_{ij}.\label{eq:ARM}
\end{equation}
In other words, the outcome $X_{ij}$ is determined by a noisy version of $\mu_{r(i)r(j)}$, which solely depends on the ranks of $i$ and $j$. In this paper, we consider $Z_{ij}$'s with sub-Gaussian tails. In particular, we assume that for any $t>0$,
\begin{equation}
\mathbb{P}\left(\sum_{1\leq i\neq j\leq n}v_{ij}Z_{ij}>t\right) \leq \exp\left(-\frac{t^2}{2\sigma^2}\right),\quad\text{for all }\sum_{1\leq i\neq j\leq n}v_{ij}^2=1.\label{eq:error}
\end{equation}

The goal of this paper is to recover the underlying ranks $r=(r(1),r(2),...,r(n))$ from the observations $\{X_{ij}\}_{1\leq i\neq j\leq n}$. We consider the following space of ranks
$$\mathcal{R}=\left\{r\in[n]^n:\left|\sum_{i=1}^nr(i)-\sum_{i=1}^n i\right|\leq c_n\right\},$$
where the number $c_n$ satisfies $1\leq c_n=o(\sqrt{n})$. The flexibility of the space $\mathcal{R}$ allows ties. To be more precise, an $r\in\mathcal{R}$ should be interpreted as discrete latent positions of the $n$ objects. Therefore, we refer to the problem as approximate ranking. A family of loss functions are considered in this paper. We define
$$
\ell_q(\hat{r},r)=\begin{cases}\frac{1}{n}\sum_{i=1}^n\mathbb{I}\{\wh{r}(i)\neq r(i)\}, & q=0, \\
\frac{1}{n}\sum_{i=1}^n|\hat{r}(i)-r(i)|^q, & q\in(0,2].
\end{cases}
$$
When $q=0$, the loss function $\ell_0(\hat{r},r)$ is the normalized Hamming distance between $\hat{r}$ and $r$. It measures the proportion of objects that are given incorrect ranks. Compared with $\ell_0(\hat{r},r)$, $\ell_q(\hat{r},r)$ with a $q\in(0,2]$ also measures the magnitude of the incorrectness of each $\hat{r}(i)$. In particular, the choice of $q=1$ leads to $\ell_1(\hat{r},r)=\frac{1}{n}\sum_{i=1}^n|\hat{r}(i)-r(i)|$, which is known to be equivalent to Kendall's tau within a factor of $2$ \citep{kumar2010generalized}. 

Our model $\mu_{r(i)r(j)}$ is quite general. It characterizes the pairwise relation between the two objects $i$ and $j$ through their ranks. The literature is popularized with pairwise comparison models. In such a setting, the value of $\mu_{r(i)r(j)}$ is an increasing function of the difference between $r(i)$ and $r(j)$. We are also interested in the pairwise collaboration setting, where a larger value of $\mu_{r(i)r(j)}$ is implied by either or both of the values of $r(i)$ and $r(j)$. Without specifying a particular setting, we impose the following general assumption. There exists a number $\beta\in\mathbb{R}$, such that for any $r,\tilde{r}\in\mathcal{R}$,
\begin{equation}
\sum_{1\leq i\neq j\leq n}\left(\mu_{\tilde{r}(i)\tilde{r}(j)}-\mu_{r(i)r(j)}\right)^2 \geq 2n\beta^2\|\tilde{r}-r\|^2.\label{eq:signal}
\end{equation}
Later in Section \ref{sec:alg}, various examples will be given to satisfy this condition.

\subsection{Minimax Rates}\label{sec:minimax}

With the observations $\{X_{ij}\}_{1\leq i\neq j\leq n}$ and the knowledge of $\mu_{r(i)r(j)}$, we consider a least-squares estimator
\begin{equation}
\hat{r}=\argmin_{r\in\mathcal{R}}\sum_{1\leq i\neq j\leq n}\left(X_{ij}-\mu_{r(i)r(j)}\right)^2.\label{eq:LSE}
\end{equation}
This estimator may not be computationally efficient and it depends on the model parameters, but it serves as an important benchmark of approximate ranking. Adaptive procedures with unknown model parameters will be discussed in Section \ref{sec:alg}.

Use $\mathbb{P}_r$ and $\mathbb{E}_r$ to denote the distribution of (\ref{eq:ARM}), and the performance of the least-squares estimator $\hat{r}$ is characterized by the following theorem.
\begin{thm}\label{thm:upper1}
Consider any loss $\ell(\hat{r},r)$ with $q\in[0,2]$.
Under the conditions (\ref{eq:error}) and (\ref{eq:signal}), we have
$$\sup_{r\in\mathcal{R}}\mathbb{E}_r\ell_q(\hat{r},r)\lesssim \begin{cases}
\exp\left(-(1+o(1))\frac{n\beta^2}{4\sigma^2}\right), & \frac{n\beta^2}{4\sigma^2}>1, \\
\left\{\left(\frac{n\beta^2}{4\sigma^2}\right)^{-1}\log\left[\left(\frac{n\beta^2}{4\sigma^2}\right)^{-1}\right]\right\}^{q/2}\wedge n^q, & \frac{n\beta^2}{4\sigma^2}\leq 1,
\end{cases}$$
where $o(1)$ denotes a vanishing sequence as $n\rightarrow\infty$.
\end{thm}

Theorem \ref{thm:upper1} reveals an important quantity of signal-to-noise ratio. It is in the form of $\frac{n\beta^2}{4\sigma^2}$. When $\frac{n\beta^2}{4\sigma^2}>1$, the ranking error converges to zero exponentially fast. In comparison, when $\frac{n\beta^2}{4\sigma^2}\leq 1$, the error has a polynomial rate depending on $q$ and capped at the order of $n^q$.

Note that the rate involves a logarithmic factor in the polynomial regime. This factor can be removed with an extra assumption on the model. Given a constant $M$ that satisfies $1<M=O(1)$, we assume that
for any $r,\tilde{r}\in\mathcal{R}$,
\begin{equation}
\sum_{1\leq i\neq j\leq n}\left(\mu_{\tilde{r}(i)\tilde{r}(j)}-\mu_{r(i)r(j)}\right)^2 \leq 2Mn\beta^2\|\tilde{r}-r\|^2.\label{eq:signal-entropy}
\end{equation}
The following theorem gives this improvement.

\begin{thm}\label{thm:upper2}
Consider any loss $\ell(\hat{r},r)$ with $q\in[0,2]$.
Under the conditions (\ref{eq:error}), (\ref{eq:signal}) and (\ref{eq:signal-entropy}), we have
$$\sup_{r\in\mathcal{R}}\mathbb{E}_r\ell_q(\hat{r},r)\lesssim \begin{cases}
\exp\left(-(1+o(1))\frac{n\beta^2}{4\sigma^2}\right), & \frac{n\beta^2}{4\sigma^2}>1, \\
\left(\frac{n\beta^2}{4\sigma^2}\right)^{-q/2}\wedge n^q, & \frac{n\beta^2}{4\sigma^2}\leq 1,
\end{cases}$$
where $o(1)$ denotes a vanishing sequence as $n\rightarrow\infty$.
\end{thm}

The rates given by Theorem \ref{thm:upper2} are sharp, and they cannot be further improved. We give matching lower bounds in the next theorem.
\begin{thm}\label{thm:lower}
Consider any loss $\ell(\hat{r},r)$ with $q\in[0,2]$.
There exists a distribution of (\ref{eq:ARM}) that satisfies (\ref{eq:error}), (\ref{eq:signal}) and (\ref{eq:signal-entropy}), such that
$$\inf_{\hat{r}}\sup_{r\in\mathcal{R}}\mathbb{E}_r\ell_q(\hat{r},r)\gtrsim \begin{cases}
\exp\left(-(1+o(1))\frac{n\beta^2}{4\sigma^2}\right), & \frac{n\beta^2}{4\sigma^2}>1, \\
\left(\frac{n\beta^2}{4\sigma^2}\right)^{-q/2}\wedge n^q, & \frac{n\beta^2}{4\sigma^2}\leq 1,
\end{cases}$$
where $o(1)$ denotes a vanishing sequence as $n\rightarrow\infty$.
\end{thm}

\subsection{Exact Recovery and Phase Transitions}\label{sec:exact}

According to Theorem \ref{thm:upper1} and Theorem \ref{thm:upper2}, the convergence rate for $\ell_q(\hat{r},r)$ is $\exp\left(-(1+o(1))\frac{n\beta^2}{4\sigma^2}\right)$ when $\frac{n\beta^2}{4\sigma^2}>1$. Therefore, suppose $\frac{n\beta^2}{4\sigma^2}>\log n$, then the convergence rate will be smaller than $n^{-1}$. Since $\ell_q(\hat{r},r)$ does not take any value in the interval $(0,n^{-1})$, a convergence rate smaller than $n^{-1}$ is expected to imply exact recover of the underlying $r$. This intuition is made rigorous by the following theorem.
\begin{thm}\label{thm:exact}
Under the conditions (\ref{eq:error}) and (\ref{eq:signal}), if we further assume that
$\liminf_n\frac{n\beta^2}{4\sigma^2\log n}>1$,
then the LSE (\ref{eq:LSE}) satisfies
$$\inf_{r\in\mathcal{R}}\mathbb{P}_r(\hat{r}=r)\rightarrow 1,$$
as $n\rightarrow\infty$.
\end{thm}
Moreover, the next result shows that the condition in Theorem \ref{thm:exact} is necessary for exact recover.
\begin{thm}\label{thm:exact-lower}
Assume $\limsup_n\frac{n\beta^2}{4\sigma^2\log n}<1$. There exists a distribution of (\ref{eq:ARM}) that satisfies (\ref{eq:error}), (\ref{eq:signal}) and (\ref{eq:signal-entropy}), such that
$$\liminf_n\inf_{\hat{r}}\sup_{r\in\mathcal{R}}\mathbb{P}_r(\hat{r}\neq r)\geq \frac{1}{2}.$$
\end{thm}

The results in Theorem \ref{thm:upper2}, Theorem \ref{thm:lower}, Theorem \ref{thm:exact} and Theorem \ref{thm:exact-lower} together give us a very good picture of the optimal error behavior. Interesting phase transitions are revealed for the approximate ranking problem. Depending on the signal-to-noise ratio $\frac{n\beta^2}{4\sigma^2}$, the optimal error rates have different behaviors. A graphical illustration is given by Figure \ref{fig:laoluan}. We summarize the phase transitions in the following table.
\begin{center}
  \begin{tabular}{ l | c |c| c|c }
    \hline
    $\ell_q(\hat{r},r)$ & trivial & non-trivial & consistent & strongly consistent \\ \hline
    $q=0$ & $ \frac{n\beta^2}{4\sigma^2}= O(1)$ & $ \frac{n\beta^2}{4\sigma^2}\gg 1$ & $ \frac{n\beta^2}{4\sigma^2}\gg 1$ & $ \frac{n\beta^2}{4\sigma^2}> \log n$ \\ \hline
    $q\in(0,2]$ & $ \frac{n\beta^2}{4\sigma^2}= O(n^{-2})$ & $ \frac{n\beta^2}{4\sigma^2}\gg n^{-2}$ & $ \frac{n\beta^2}{4\sigma^2}\gg 1$ & $ \frac{n\beta^2}{4\sigma^2}> \log n$\\ \hline
  \end{tabular}
\end{center}
We call a rate trivial if it is at the same order of the maximal value of the loss. For example, the maximum of the three loss functions $\ell_0(\hat{r},r)$, $\ell_1(\hat{r},r)$ and $\ell_2(\hat{r},r)$ are of orders $1$, $n$ and $n^2$, respectively. A rate is consistent if the loss converges to zero. We refer to exact recovery as being strongly consistent. Then, for each loss function, there are four different regimes: the trivial regime, the non-trivial regime, the consistent regime and the strongly consistent regime. The only exception is that for the loss $\ell_0(\hat{r},r)$, its non-trivial regime is identical to the consistent regime. For all other loss functions with $q\in(0,2]$, the rates in the consistent regime are exponential and the rates in the non-trivial regime are polynomial.

\section{Adaptive Procedures}\label{sec:alg}

Our paper considers a very general framework of the approximate ranking model in the form of $\mu_{r(i)r(j)}$. Though we are able to characterize the exact minimax rates of the problem under various regimes of the signal-to-noise ratio, the least-squares procedure (\ref{eq:LSE}) that can achieve the statistical optimality is usually infeasible in practice. In fact, similar optimization problems as (\ref{eq:LSE}) have been considered in the literature of graph matching/isomorphism problem \citep{zaslavskiy2009path}, which is believed very unlikely to be solved by a polynomial-time algorithm \citep{kobler2012graph}.

In this section, we consider some special cases of the general model $\mu_{r(i)r(j)}$, and then discuss possible adaptive procedures that can take advantage of the additional model structures to achieve the optimal statistical rates. Inspired by the latent space model in network analysis \citep{hoff2002latent}, we consider examples in the form
\begin{equation}
\mu_{r(i)r(j)}=f\left(\theta_{r(i)},\theta_{r(j)}\right).\label{eq:latent-ability}
\end{equation}
Here, $\theta_{r(i)}$ stands for the latent ability parameter of the position $r(i)$. In particular, we will analyze a differential comparison model and an additive collaboration model, both of which are in the form of (\ref{eq:latent-ability}). Interestingly, we will show for these two models, the approximate ranking problem is reduced to a feature matching problem considered by \cite{collier2016minimax}, for which efficient algorithms are available. When the latent parameters $\{\theta_i\}$ are unknown but admits a linear relation with respect to the underlying ranks, we show that a profile least-squares procedure, which can be solved by iterative feature matching, can adaptively achieve the optimal statistical rates.

\subsection{Differential Comparison}\label{sec:comp}

We first consider a differential comparison model, which imposes the structure
\begin{equation}
\mu_{ij}=\theta_i-\theta_j.\label{eq:def-comp}
\end{equation}
Therefore, the mean of the observation $X_{ij}$ is given by $\theta_{r(i)}-\theta_{r(j)}$, the difference between the abilities of $i$ and $j$. We propose the following signal conditions. For any $r,\tilde{r}\in\mathcal{R}$, we assume
\begin{equation}
\label{eq:signal-comp}\sum_{i=1}^n(\theta_{r(i)}-\theta_{\tilde{r}(i)})^2 \geq \beta^2\|\tilde{r}-r\|^2,
\end{equation}
\begin{equation}
\label{eq:stable-comp}\max_{r\in\mathcal{R}}\left|\sum_{i=1}^n\theta_{r(i)}-\sum_{i=1}^n\theta_i\right| = o(\sqrt{n}\beta).
\end{equation}
It is easy to check that the general condition (\ref{eq:signal}) is implied by (\ref{eq:signal-comp}) and (\ref{eq:stable-comp}). We remark that the condition (\ref{eq:stable-comp}) is coherent with the definition of the space $\mathcal{R}$.

In the current setting, the ability parameters $\{\theta_i\}$ are given but the correspondence between $\{i\}$ and $\{\theta_{r(i)}\}$ is linked by unknown ranks $\{r(i)\}$. Our general strategy to find ranks is based on the idea of feature matching \citep{collier2016minimax}. We first define the score of $i$ by
\begin{equation}
S_i=\frac{1}{2n}\sum_{j\in[n]\backslash\{i\}}(X_{ij}-X_{ji})+\frac{1}{n}\sum_{j=1}^n\theta_j.\label{eq:score-comp}
\end{equation}
Then, the estimator of ranks is defined by
\begin{equation}
\hat{r}=\argmin_{r\in\mathcal{R}}\sum_{i=1}^n(S_i-\theta_{r(i)})^2.\label{eq:feature-matching}
\end{equation}
This optimization can be efficiently solved by feature matching algorithms discussed in \cite{collier2016minimax}. Its statistical performance is given by the following theorem.
\begin{thm}\label{thm:upper-comp}
Consider any loss $\ell(\hat{r},r)$ with $q\in[0,2]$.
Under the conditions (\ref{eq:error}), (\ref{eq:signal-comp}) and (\ref{eq:stable-comp}), we have
$$\sup_{r\in\mathcal{R}}\mathbb{E}_r\ell_q(\hat{r},r)\lesssim \begin{cases}
\exp\left(-(1+o(1))\frac{n\beta^2}{4\sigma^2}\right), & \frac{n\beta^2}{4\sigma^2}>1, \\
\left(\frac{n\beta^2}{4\sigma^2}\right)^{-q/2}\wedge n^q, & \frac{n\beta^2}{4\sigma^2}\leq 1,
\end{cases}$$
where $o(1)$ denotes a vanishing sequence as $n\rightarrow\infty$. Moreover, when $\liminf_n\frac{n\beta^2}{4\sigma^2\log n}>1$, we have $\hat{r}=r$ with probability $1-o(1)$.
\end{thm}
It is easy to check that the distribution constructed to prove the lower bound results in Theorem \ref{thm:lower} satisfies the conditions (\ref{eq:signal-comp}) and (\ref{eq:stable-comp}). This implies that the rates achieved by the computationally efficient estimator $\hat{r}$ in Theorem \ref{thm:upper-comp} is optimal under the differential comparison model. It is interesting that we do not have any logarithmic factor in the convergence rates when $\frac{n\beta^2}{4\sigma^2}\leq 1$ even without any condition similar to (\ref{eq:signal-entropy}). This is a consequence by the special structure in the differential comparison model (\ref{eq:def-comp}).

\subsection{Additive Collaboration}\label{sec:coll}

Since our framework is quite general, we can also consider an additive collaboration model. It imposes the structure
\begin{equation}
\mu_{ij}=\theta_i+\theta_j.\label{eq:def-add}
\end{equation}
Compared with (\ref{eq:def-comp}), the collaboration model assumes that the mean $\mu_{ij}$ is increasing with respect to both the abilities of $i$ and $j$. Despite the difference in interpretation, the two models have very similar mathematical structures. We will adopt the signal condition (\ref{eq:signal-comp}) here, but we do not need to assume the second condition (\ref{eq:stable-comp}). With the additive structure, the condition (\ref{eq:signal-comp}) alone implies (\ref{eq:signal}).  We also use the feature-matching estimator (\ref{eq:feature-matching}), with the score of each $i$ defined as
\begin{equation}
S_i=\frac{1}{2(n-2)}\left(\sum_{j\in[n]\backslash\{i\}}(X_{ij}+X_{ji})-\frac{1}{n-1}\sum_{1\leq i\neq j\leq n}X_{ij}\right).\label{eq:score-coll}
\end{equation}
The performance of the estimator $\hat{r}$ is given by the following theorem.
\begin{thm}\label{thm:upper-add}
Consider any loss $\ell(\hat{r},r)$ with $q\in[0,2]$.
Under the conditions (\ref{eq:error}) and (\ref{eq:signal-comp}), we have
$$\sup_{r\in\mathcal{R}}\mathbb{E}_r\ell_q(\hat{r},r)\lesssim \begin{cases}
\exp\left(-(1+o(1))\frac{n\beta^2}{4\sigma^2}\right), & \frac{n\beta^2}{4\sigma^2}>1, \\
\left(\frac{n\beta^2}{4\sigma^2}\right)^{-q/2}\wedge n^q, & \frac{n\beta^2}{4\sigma^2}\leq 1,
\end{cases}$$
where $o(1)$ denotes a vanishing sequence as $n\rightarrow\infty$. Moreover, when $\liminf_n\frac{n\beta^2}{4\sigma^2\log n}>1$, we have $\hat{r}=r$ with probability $1-o(1)$.
\end{thm}
We obtain the same rates as in Theorem \ref{thm:upper-comp} for the differential comparison model. Note that the same argument in the lower bound proof of Theorem \ref{thm:lower} can easily be adapted for the additive collaboration model in this section. This implies the rates in Theorem \ref{thm:upper-add} are all optimal.

\subsection{Applications in a Parametric Model}\label{sec:para}

In this section, we consider both the comparison and the collaboration models with $\theta_i=\alpha+\tilde{\beta}i$ for some $\alpha\in\mathbb{R}$ and $\tilde{\beta}>0$. That is, the ability parameter $\theta_i$ is a linear function of its latent position. It is easy to check that both conditions (\ref{eq:signal-comp}) and (\ref{eq:stable-comp}) are satisfied with some $\beta=(1+o(1))\tilde{\beta}$.

Since $\alpha$ and $\tilde{\beta}$ are unknown, we cannot directly use the feature matching estimator (\ref{eq:feature-matching}). Instead, we propose the following profile least-squares objective,
\begin{equation}
\text{PL}(r)=\min_{a,b}\sum_{i=1}^n\left(\hat{S}_i-a-br(i)\right)^2.\label{eq:PL}
\end{equation}
Then, the estimator is found through minimizing $\text{PL}(r)$. Note that we use a linear model $a+br(i)$ as a surrogate for $\theta_{r(i)}$ in the definition of (\ref{eq:PL}). The feature matching procedure and the linear model fit of $\theta_{r(i)}$ are done simultaneously. An equivalent way of writing $\text{PL}(r)$ is by
$$\text{PL}(r)=\sum_{i=1}^n\left(\hat{S}_i-\hat{a}_r-\hat{b}_rr(i)\right)^2,$$
where
\begin{equation}
\hat{b}_r=\frac{\frac{1}{n}\sum_{i=1}^n\hat{S}_ir(i)-\left(\frac{1}{n}\sum_{i=1}^nr(i)\right)\left(\frac{1}{n}\sum_{i=1}^nS_i\right)}{\frac{1}{n}\sum_{i=1}^nr(i)^2-\left(\frac{1}{n}\sum_{i=1}^nr(i)\right)^2}\quad\text{and}\quad \hat{a}_r=\frac{1}{n}\sum_{i=1}^nS_i-\hat{b}_r\frac{1}{n}\sum_{i=1}^nr(i).\label{eq:OLS}
\end{equation}
For the differential comparison model, we use the score
$$\hat{S}_i=\frac{1}{2n}\sum_{j\in[n]\backslash\{i\}}(X_{ij}-X_{ji}).$$
For the additive collaboration model, we use
$$\hat{S}_i=\frac{1}{2(n-2)}\sum_{j\in[n]\backslash\{i\}}(X_{ij}+X_{ji}).$$
Therefore, the estimator $\hat{r}$ is fully data-driven.

Optimizing over $\text{PL}(r)$ can be done in a iterative fashion. At each iteration, one first optimize $\sum_{i=1}^n\left(\hat{S}_i-\hat{a}_{r_{t-1}}-\hat{b}_{r_{t-1}}r(i)\right)^2$ over $r$. Then, one can update the parameters $\hat{a}_{r_{t-1}}$ and $\hat{b}_{r_{t-1}}$ using the least-squares formula (\ref{eq:OLS}). In other words, feature matching and linear regression are run in turn iteratively.
In this section, our focus is on the study of the statistical property of the global optimizer of $\text{PL}(r)$. The convergence analysis of the iterative algorithm will be studied in a much more general framework of profile least-squares optimization in the future.

To study the statistical error of the profile least-squares estimator, we consider the following space of approximate ranks. Define
$$\mathcal{R}'=\left\{r\in\mathcal{R}: \left|\sum_{i=1}^nr(i)^2-\sum_{i=1}^n i^2\right|\leq c_n'\right\},$$
where $c_n'=o(n^3)$. The set $\mathcal{R}'$ is a subset of $\mathcal{R}$ with an additional constraint on $\sum_{i=1}^nr(i)^2$. This extra constraint does not make the problem easier, and the same minimax lower bound results also hold for the set $\mathcal{R}'$ with a simple modification of the proof of Theorem \ref{thm:lower}.

\begin{thm}\label{thm:PL}
Consider the estimator $\hat{r}=\argmin_{r\in\mathcal{R}'}\text{PL}(r)$. Under the conditions (\ref{eq:error}) for the differential comparison model or the additive collaboration model with $\theta_i=\alpha+\tilde{\beta}i$ for some $\tilde{\beta}=(1+o(1))\beta$, we have
$$\sup_{r\in\mathcal{R}'}\mathbb{E}_r\ell_q(\hat{r},r)\lesssim \begin{cases}
\exp\left(-(1+o(1))\frac{n\beta^2}{4\sigma^2}\right), & \frac{n\beta^2}{4\sigma^2}>1, \\
\left(\frac{n\beta^2}{4\sigma^2}\right)^{-q/2}\wedge n^q, & \frac{n\beta^2}{4\sigma^2}\leq 1,
\end{cases}$$
where $o(1)$ denotes a vanishing sequence as $n\rightarrow\infty$. Moreover, when $\liminf_n\frac{n\beta^2}{4\sigma^2\log n}>1$, we have $\hat{r}=r$ with probability $1-o(1)$.
\end{thm}

\section{Discussion}\label{sec:disc}

\subsection{Comparison with Community Detection}

Our approximate ranking model shares some similarity with the stochastic block model that is widely studied in the problem of community detection. The relation between ranking and clustering was previously discussed in the paper \cite{chen2016information}. Our discussion here is from a different aspect. The goal of community detection is to partition the set $[n]$ into $k$ clusters. In the setting of stochastic block model, the observation can be written as
\begin{equation}
X_{ij}=\mu_{\tau(i)\tau(j)}+Z_{ij},\label{eq:SBM}
\end{equation}
where $\mu_{\tau(i)\tau(j)}$ is the expectation of $X_{ij}$, and it characterizes the interaction level between $i$ and $j$. The value of $\mu_{\tau(i)\tau(j)}$ is determined by the clustering labels $\tau(i)$ and $\tau(j)$. Note that the form (\ref{eq:SBM}) is exactly the same as (\ref{eq:ARM}). The literature usually studies stochastic block models with Bernoulli observations. However, to make comparison more intuitive, we consider the same sub-Gaussian setting as in (\ref{eq:error}).

Just like in the approximate ranking model, the goal here is to estimate the clustering labels $\tau\in\mathcal{C}_k$. In the most basic setting, the class $\mathcal{C}_k$ is considered as a subset of $[k]^n$ that consists of clustering labels that induce roughly equal-sized clusters. Recently, the minimax rate of estimating $\tau$ was derived by \cite{zhang2016minimax}. They impose the condition that
\begin{equation}
\mu_{\tau(i)\tau(j)}=p\mathbb{I}\{\tau(i)=\tau(j)\}+q\mathbb{I}\{\tau(i)\neq \tau(j)\}.\label{eq:zhang}
\end{equation}
The numbers $p$ and $q$ represent the within-cluster and the between-cluster interaction levels. In fact, we can write down an alternative condition in the style of (\ref{eq:signal}). Assume that for any $\tau,\tilde{\tau}\in\mathcal{C}_k$,
\begin{equation}
\sum_{1\leq i\neq j\leq n}(\mu_{\tau(i)\tau(j)}-\mu_{\tilde{\tau}(i)\tilde{\tau}(j)})^2 \geq 2\beta^2n^2\tilde{\ell}_0(\tau,\tilde{\tau}),\label{eq:gao}
\end{equation}
when $\tilde{\ell}_0(\tau,\tilde{\tau})=o(n/k)$. The loss function $\tilde{\ell}_0(\tau,\tilde{\tau})$ is defined in the same way as $\ell_0(\tilde{\tau},\tau)$ up to a permutation of clustering labels, so that it is more appropriate to measure the difference between two clustering structures. One can check that (\ref{eq:zhang}) implies (\ref{eq:gao}) with $\beta^2=(1+o(1))\frac{2(p-q)^2}{k}$. With similar techniques used in this paper, it can be shown that there is an estimator $\hat{\tau}$ that achieves
$$\mathbb{E}_{\tau}\tilde{\ell}_0(\hat{\tau},\tau)\leq \exp\left(-(1+o(1))\frac{n(p-q)^2}{4k\sigma^2}\right).$$
This is essentially the same result in \cite{zhang2016minimax} by replacing $\sigma^2$ with the variance parameter $\left(\frac{\sqrt{p}+\sqrt{q}}{2}\right)^2$ in the their Bernoulli setting. Moreover, it shares the same form of exponential rates in Theorem \ref{thm:upper1} in view of the relation $\beta^2=(1+o(1))\frac{2(p-q)^2}{k}$.

On the other hand, we also point out some major differences between approximate ranking and community detection. First of all, the ranking labels are ordered numbers, while the clustering labels have no ordering. Therefore, one can only measure whether the estimation of $\tau(i)$ is right or wrong by an indicator function. In comparison, one can not only measure whether each $r(i)$ is correctly estimated, but one can also measure the deviation in terms of the distance or the squared distance between $\hat{r}(i)$ and $r(i)$. Secondly, the approximate ranking model naturally has $n$ latent positions, and each $r(i)$ has $n$ possible values, while for the stochastic block model, there are only $k$ latent positions, and usually $k$ is assumed to be a constant or grow with $n$ very slowly in the literature. These two differences lead to the interesting phase transitions for the approximate ranking problem in our paper, and such a new phenomenon did not appear in community detection or in any other problems before.

\subsection{Results for Poisson Distributions}

In this section, we consider a natural Poisson model for discrete observations. We assume $X_{ij}\sim\text{Poisson}(\mu_{r(i)r(j)})$ independently for all $1\leq i\neq j\leq n$. Note that $\mu_{r(i)r(j)}$ models both mean and variance of the observation $X_{ij}$. Thus, it is more appropriate to consider the following condition instead of (\ref{eq:signal}). We assume there exists a $\beta\in\mathbb{R}$, such that for any $r,\tilde{r}\in[n]^n$,
\begin{equation}
\sum_{1\leq i\neq j\leq n}\left(\sqrt{\mu_{\tilde{r}(i)\tilde{r}(j)}}-\sqrt{\mu_{r(i)r(j)}}\right)^2 \geq 2n\beta^2\|\tilde{r}-r\|^2. \label{eq:signal-Poi}
\end{equation}
Compared with the condition (\ref{eq:signal}), the condition for the Poisson model involves $\sqrt{\mu_{r(i)r(j)}}$. The square-root transformation can be dated back to the famous variance-stabilizing transformation \cite{anscombe1948transformation}.

Instead of using the least-squares estimator (\ref{eq:LSE}), we propose the maximum likelihood estimator
$$\hat{r}=\argmax_{r\in\mathcal{R}}\prod_{1\leq i\neq j\leq n}\frac{\mu_{r(i)r(j)}^{X_{ij}}e^{-\mu_{r(i)r(j)}}}{X_{ij}!}.$$
\begin{thm}\label{thm:upper-Poi}
Consider any loss $\ell(\hat{r},r)$ with $q\in[0,2]$.
Under the condition (\ref{eq:signal-Poi}), we have
$$\sup_{r\in\mathcal{R}}\mathbb{E}_r\ell_q(\hat{r},r)\lesssim \begin{cases}
\exp\left(-(1+o(1))n\beta^2\right), & n\beta^2>1, \\
\left[\left(n\beta^2\right)^{-1}\log\left[\left(n\beta^2\right)^{-1}\right]\right]^{q/2}\wedge n^q, & n\beta^2\leq 1,
\end{cases}$$
where $o(1)$ denotes a vanishing sequence as $n\rightarrow\infty$.
\end{thm}
Again, Theorem \ref{thm:upper-Poi} exhibits different behaviors of the ranking errors depending on the regime of $n\beta^2$. Here, the signal-to-noise ratio is $n\beta^2$, and it plays the same role as that of $\frac{n\beta^2}{4\sigma^2}$ in Theorem \ref{thm:upper1}.

We also give a complementary lower bound.
\begin{thm}\label{thm:lower-Poi}
Consider any loss $\ell(\hat{r},r)$ with $q\in[0,2]$.
There exists a Poisson distribution that satisfies (\ref{eq:signal-Poi}), such that
$$\inf_{\hat{r}}\sup_{r\in\mathcal{R}}\mathbb{E}_r\ell_q(\hat{r},r)\gtrsim \begin{cases}
\exp\left(-(1+o(1))n\beta^2\right), & n\beta^2>1, \\
\left(n\beta^2\right)^{-q/2}\wedge n^q, & n\beta^2\leq 1,
\end{cases}$$
where $o(1)$ denotes a vanishing sequence as $n\rightarrow\infty$.
\end{thm}

The lower bound results imply that the rates that we obtain in Theorem \ref{thm:upper-Poi} are optimal up to a logarithmic factor in the polynomial regime where $n\beta^2\leq 1$. Different from what we have for the model (\ref{eq:ARM}), it is not clear whether such a logarithmic factor can be removed in the upper bounds with some extra condition for the Poisson model.

To close this section, we give sufficient and necessary conditions for exact recovery in the following theorem.
\begin{thm}\label{thm:exact-Poi}
Under the condition (\ref{eq:signal-Poi}), if we further assume that $\liminf_n\frac{n\beta^2}{\log n}>1$, then the MLE satisfies
$$\inf_{r\in\mathcal{R}}\mathbb{P}_r(\hat{r}=r)\rightarrow 1,$$
as $n\rightarrow\infty$. Moreover, when $\limsup_n\frac{n\beta^2}{\log n}<1$, there exists a Poisson distribution that satisfies (\ref{eq:signal-Poi}), such that
$$\liminf_n\inf_{\hat{r}}\sup_{r\in\mathcal{R}}\mathbb{P}_r(\hat{r}\neq r)\geq \frac{1}{2}-o(1).$$
\end{thm}

\subsection{Other Ranking Models}

Our paper gives a general analysis of the optimal rates of the approximate ranking model (\ref{eq:ARM}). Adaptive procedures are discussed for two special cases of differential comparison and additive collaboration. We expect the analysis can be extended to derive exact optimal rates and phase transitions for many other models. For example, the popular Bradley-Terry-Luce model \citep{bradley1952rank,luce2005individual} considers the form $\mu_{r(i)r(j)}=\frac{\theta_{r(i)}}{\theta_{r(i)}+\theta_{r(j)}}$, a special case of (\ref{eq:latent-ability}). Ranking problems under this setting were studied by \cite{negahban2012iterative} and references therein. 
Besides the parametric models, an interesting nonparametric stochastically transitive model was proposed and analyzed in \cite{shah2016stochastically,shah2015simple}. Finally, a simple noisy sorting model that assumes $\mu_{r(i)r(j)}\geq \frac{1}{2}+\gamma$ if $r(i)<r(j)$ was considered in \cite{braverman2008noisy}. The minimax rate for this model was recently derived by \cite{mao2017minimax}.

We remark that all of these models proposed in the literature can be written in some modified versions of (\ref{eq:signal}). However, since these papers all consider Bernoulli observations and the space of permutations, our noise condition (\ref{eq:error}) and the setting of approximate ranking may not be appropriate. Despite these differences, we still believe all the phase transition boundaries discovered in our paper have analogous correspondence in these Bernoulli models. This will be an immediate interesting project to follow in the future.

\section{Proofs}\label{sec:proof}

In this section, we give proofs of all the results in the paper. From Section \ref{sec:pf-thm-first} to Section \ref{sec:pf-thm-last}, we state the proofs of the main results, with the help of some technical lemmas. The proof of these auxiliary lemmas will be given in Section \ref{sec:pf-aux}.

\subsection{Proofs of Theorem \ref{thm:upper1} and Theorem \ref{thm:upper2}}\label{sec:pf-thm-first}

Before stating the proofs, we need some lemmas.
We use $L(r)$ to denote the objective function $\sum_{i\neq j}(X_{ij}-\mu_{r(i)r(j)})^2$. For a fixed $r$, we define
\begin{equation}
\mathcal{R}_m=\left\{\tilde{r}\in\mathcal{R}:\|\tilde{r}-r\|^2=m\right\}.\label{eq:def-rm}
\end{equation}
In other words, the set $\mathcal{R}_m$ collects those $\tilde{r}$'s that have errors $m$ in terms of the squared $\ell_2$ norm. This results in a partition of $\mathcal{R}$, which is
$$\mathcal{R}=\cup_m\mathcal{R}_m.$$
The following two lemmas are important to prove the main results, and their proofs will be given in Section \ref{sec:pf-aux}.
\begin{lemma}\label{lem:dev}
Assume the conditions (\ref{eq:error}) and (\ref{eq:signal}).
For any $m$ such that $\mathcal{R}_m\neq\varnothing$, we have
$$\max_{\tilde{r}\in\mathcal{R}_m}\mathbb{P}_r\left(L(\tilde{r})\leq L(r)\right)\leq\exp\left(-\frac{n\beta^2m}{4\sigma^2}\right).$$
\end{lemma}
\begin{lemma}\label{lem:card}
For each $m$, the cardinality of $\mathcal{R}_m$ is bounded as
$$|\mathcal{R}_m|\leq \begin{cases}
\left(\frac{2e^2n}{m}\right)^m, & 1\leq m\leq n,\\
\left(\frac{8em}{n}\right)^n, & n\leq m\leq n^2, \\
n^n, & m>n^2.
\end{cases}$$
\end{lemma}

Now we are ready to prove Theorem \ref{thm:upper1}.
\begin{proof}[Proof of Theorem \ref{thm:upper1}]
We first give a bound for $\mathbb{E}_r\|\hat{r}-r\|^2$. Different arguments will be given depending on the value of $\frac{n\beta^2}{4\sigma^2}$.
\paragraph{Case 1: $\frac{n\beta^2}{4\sigma^2}\geq 2\log (2e^2n)$.}
In this regime, we have
\begin{eqnarray}
\nonumber \mathbb{E}_r\|\hat{r}-r\|^2 &\leq& \sum_{m=1}^{n^3}m\mathbb{P}_r(\|\hat{r}-r\|^2=m) \\
\nonumber &\leq& \sum_{m=1}^{n^3}m\sum_{\tilde{r}\in\mathcal{R}_m}\mathbb{P}_r\left(L(\tilde{r})\leq L(r)\right) \\
\label{eq:bd-error} &\leq& \sum_{m=1}^{n^3}m|\mathcal{R}_m|\exp\left(-\frac{n\beta^2m}{4\sigma^2}\right) \\
\label{eq:bd-card} &\leq& \sum_{m=1}^nm\left(\frac{2e^2n}{m}\right)^m\exp\left(-\frac{n\beta^2m}{4\sigma^2}\right) + n^n\sum_{m=n+1}^{n^3}m\exp\left(-\frac{n\beta^2m}{4\sigma^2}\right).
\end{eqnarray}
In the above derivation, we use Lemma \ref{lem:dev} for (\ref{eq:bd-error}) and Lemma \ref{lem:card} for (\ref{eq:bd-card}).
Now we will give bounds for the two terms in (\ref{eq:bd-card}) separately. The first term has bound
$$\sum_{m=1}^n(2e^2n)^m\exp\left(-\frac{n\beta^2m}{4\sigma^2}\right)=\sum_{m=1}^n\exp\left(-m\left[\frac{n\beta^2}{4\sigma^2}-\log(2e^2n)\right]\right).$$
Under the condition $\frac{n\beta^2}{4\sigma^2}\geq 2\log (2e^2n)$, we have 
$$\sum_{m=1}^n\exp\left(-m\left[\frac{n\beta^2}{4\sigma^2}-\log(2e^{2}n)\right]\right)\lesssim \exp\left(-\left[\frac{n\beta^2}{4\sigma^2}-\log(2e^{2}n)\right]\right)=n\exp\left(-(1+o(1))\frac{n\beta^2}{4\sigma^2}\right).$$
The second term of (\ref{eq:bd-card}) has bound
\begin{eqnarray*}
n^n\sum_{m=n+1}^{n^3}m\exp\left(-\frac{n\beta^2m}{4\sigma^2}\right) &\leq& n^nn^6\exp\left(-\frac{n^2\beta^2}{4\sigma^2}\right) \\
&\leq& \exp\left(-n\left[\frac{n\beta^2}{4\sigma^2}-7\log(n)\right]\right) \\
&\leq& \exp\left(-\left[\frac{n\beta^2}{4\sigma^2}-7\log(n)\right]\right) \\
&=& n\exp\left(-(1+o(1))\frac{n\beta^2}{4\sigma^2}\right).
\end{eqnarray*}
Therefore, in this regime, we have
$$\mathbb{E}_r\|\hat{r}-r\|^2\lesssim n\exp\left(-(1+o(1))\frac{n\beta^2}{4\sigma^2}\right).$$

\paragraph{Case 2: $2\log(16e)<\frac{n\beta^2}{4\sigma^2}<2\log (2e^{2}n)$.}
In this regime, we have
\begin{eqnarray}
\nonumber \mathbb{E}_r\|\hat{r}-r\|^2 &\leq& m_0 + \sum_{m>m_0}\mathbb{P}_r\left(\|\hat{r}-r\|^2\geq m\right) \\
\nonumber &\leq& m_0 + \sum_{m>m_0}m\mathbb{P}_r\left(\|\hat{r}-r\|^2= m\right) \\
\label{eq:bd-error2} &\leq& m_0 + \sum_{m>m_0}m|\mathcal{R}_m|\exp\left(-\frac{n\beta^2m}{4\sigma^2}\right) \\
\label{eq:bd-card2} &\leq& m_0 + \sum_{m=m_0+1}^nm\left(\frac{2e^{2}n}{m}\right)^m\exp\left(-\frac{n\beta^2m}{4\sigma^2}\right) \\
\nonumber && + \sum_{m=n+1}^{n^2}m\left(\frac{8em}{n}\right)^n\exp\left(-\frac{n\beta^2m}{4\sigma^2}\right) + n^n\sum_{m>n^2}m\exp\left(-\frac{n\beta^2m}{4\sigma^2}\right).
\end{eqnarray}
Again, Lemma \ref{lem:dev} and Lemma \ref{lem:card} are used to derive (\ref{eq:bd-error2}) and (\ref{eq:bd-card2}). There are four terms in (\ref{eq:bd-card2}) that we need to bound.
For the first term, we take
$$m_0=2\exp\left(-\frac{n\beta^2}{4\sigma^2}+2\log(4e^{2}n)\right)\lesssim n\exp\left(-(1+o(1))\frac{n\beta^2}{4\sigma^2}\right).$$
We remark that for this choice of $m_0$, we have $m_0\gtrsim 1$ under the condition $\frac{n\beta^2}{4\sigma^2}<2\log (2e^{2}n)$.
Then, the second term can be bounded by
\begin{eqnarray*}
&& \sum_{m=m_0+1}^nm\left(\frac{2e^{2}n}{m}\right)^m\exp\left(-\frac{n\beta^2m}{4\sigma^2}\right) \\
&\leq& \sum_{m=m_0+1}^n\left(\frac{4e^{2}n}{m}\right)^m\exp\left(-\frac{n\beta^2m}{4\sigma^2}\right) \\
&\leq& \sum_{m\geq m_0+1}\left(\frac{\exp\left(-\frac{n\beta^2}{4\sigma^2}+\log(4e^{2}n)\right)}{m_0}\right)^m \\
&\leq& \sum_{m\geq m_0+1}2^{-m}\lesssim 1\lesssim m_0.
\end{eqnarray*}
The third term can be bounded by
\begin{eqnarray}
\nonumber && \sum_{m=n+1}^{n^2}m\left(\frac{8em}{n}\right)^n\exp\left(-\frac{n\beta^2m}{4\sigma^2}\right) \\
\nonumber &\leq& \sum_{m=n+1}^{n^2}\left(\frac{16em}{n}\right)^n\exp\left(-\frac{n\beta^2m}{4\sigma^2}\right) \\
\label{eq:tbc} &\leq& \sum_{m>n}\exp\left(-\frac{mn\beta^2}{8\sigma^2}\right)\left[\exp\left(-\frac{n\beta^2}{8\sigma^2}+\frac{n}{m}\log\left(\frac{16em}{n}\right)\right)\right]^m.
\end{eqnarray}
Note that under the condition $2\log(16e)<\frac{n\beta^2}{4\sigma^2}$,
$$-\frac{n\beta^2}{8\sigma^2}+\frac{n}{m}\log\left(\frac{16em}{n}\right)\leq -\frac{n\beta^2}{8\sigma^2}+\log(16e)<0,$$
for all $m>n$. Thus, (\ref{eq:tbc}) can be bounded by
$$\sum_{m>n}\exp\left(-\frac{mn\beta^2}{8\sigma^2}\right)\lesssim \exp\left(-\frac{n^2\beta^2}{8\sigma^2}\right)\lesssim 1\lesssim m_0.$$
For the last term, we have the bound
\begin{eqnarray}
\nonumber n^n\sum_{m>n^2}m\exp\left(-\frac{n\beta^2m}{4\sigma^2}\right) &\leq& n^n\sum_{m>n^2}\exp\left(-m\left[\frac{n\beta^2}{4\sigma^2}-1\right]\right) \\
\nonumber &\lesssim& n^n\exp\left(-n^2\left[\frac{n\beta^2}{4\sigma^2}-1\right]\right) \\
\label{eq:3rd-bd} &\lesssim& 1\lesssim m_0.
\end{eqnarray}
Hence, in this regime, we have
$$\mathbb{E}_r\|\hat{r}-r\|^2\lesssim n\exp\left(-(1+o(1))\frac{n\beta^2}{4\sigma^2}\right).$$

\paragraph{Case 3: $\frac{\log n}{n}<\frac{n\beta^2}{4\sigma^2}\leq 2\log(16e)$.}
For the this regime, we use a similar argument that is used in the previous two, and we obtain the following bound
$$\mathbb{E}_r\|\hat{r}-r\|^2\leq m_0 + \sum_{m=m_0+1}^{n^2}m\left(\frac{8em}{n}\right)^n\exp\left(-\frac{n\beta^2m}{4\sigma^2}\right) + n^n\sum_{m=n^2+1}^{n^3}m\exp\left(-\frac{n\beta^2m}{4\sigma^2}\right).$$
There are three terms to bound. For the first term,
we choose
$$m_0=Cn\left(\frac{n\beta^2}{4\sigma^2}\right)^{-1}\log\left[\left(\frac{n\beta^2}{4\sigma^2}\right)^{-1}\right],$$
for some large constant $C>0$. Note that $m_0\geq n$ under the condition $\frac{n\beta^2}{4\sigma^2}\leq 2\log(16e)$. Then, the second term is bounded by
\begin{eqnarray*}
&& \sum_{m=m_0+1}^{n^2}m\left(\frac{8em}{n}\right)^n\exp\left(-\frac{n\beta^2m}{4\sigma^2}\right) \\
&\leq& \sum_{m>m_0}\exp\left(-\frac{mn\beta^2}{8\sigma^2}\right)\left[\exp\left(-\frac{n\beta^2}{8\sigma^2}+\frac{n}{m_0}\log\left(\frac{16em_0}{n}\right)\right)\right]^m.
\end{eqnarray*}
By the choice of $m_0$, $-\frac{n\beta^2}{8\sigma^2}+\frac{n}{m_0}\log\left(\frac{16em_0}{n}\right)\leq 0$ for some sufficiently large $C>0$. This gives

$$\sum_{m>m_0}\exp\left(-\frac{mn\beta^2}{8\sigma^2}\right)\left[\exp\left(-\frac{n\beta^2}{8\sigma^2}+\frac{n}{m_0}\log\left(\frac{16em_0}{n}\right)\right)\right]^m\leq \sum_{m>m_0}\exp\left(-\frac{mn\beta^2}{8\sigma^2}\right)\lesssim m_0.$$
The third term can be bounded in the same way as the bound (\ref{eq:3rd-bd}). Hence, in this regime, we have
$$\mathbb{E}_r\|\hat{r}-r\|^2\lesssim n\left(\frac{n\beta^2}{4\sigma^2}\right)^{-1}\log\left[\left(\frac{n\beta^2}{4\sigma^2}\right)^{-1}\right].$$

\paragraph{Case 4: $\frac{n\beta^2}{4\sigma^2}\leq \frac{\log n}{n}$.}
For this regime, we have the bound
$$\mathbb{E}_r\|\hat{r}-r\|^2\leq m_0 + n^n\sum_{m=m_0+1}^{n^3}m\exp\left(-\frac{n\beta^2m}{4\sigma^2}\right).$$
Choose
$$m_0=Cn\left(\frac{n\beta^2}{4\sigma^2}\right)^{-1}\log n,$$
for some large constant $C>0$. Then, we can bound both terms. This gives
$$\mathbb{E}_r\|\hat{r}-r\|^2\lesssim n\left(\frac{n\beta^2}{4\sigma^2}\right)^{-1}\log\left[\left(\frac{n\beta^2}{4\sigma^2}\right)^{-1}\right],$$
by using the condition $\frac{n\beta^2}{4\sigma^2}\leq \frac{\log n}{n}$.

Finally, we can summarize our results obtained in the four regimes above. We note that when $\frac{n\beta^2}{4\sigma^2}$ is at a constant level, both the bound $n\exp\left(-(1+o(1))\frac{n\beta^2}{4\sigma^2}\right)$ and the bound $n\left(\frac{n\beta^2}{4\sigma^2}\right)^{-1}\log\left[\left(\frac{n\beta^2}{4\sigma^2}\right)^{-1}\right]$ are of order $n$, and thus they can be used interchangeably. With the relation $\ell_2(\hat{r},r)=\|\hat{r}-r\|^2/n$, we have
$$\mathbb{E}_r\ell_2(\hat{r},r)\lesssim \begin{cases}
\exp\left(-(1+o(1))\frac{n\beta^2}{4\sigma^2}\right), & \frac{n\beta^2}{4\sigma^2}>1, \\
\left\{\left(\frac{n\beta^2}{4\sigma^2}\right)^{-1}\log\left[\left(\frac{n\beta^2}{4\sigma^2}\right)^{-1}\right]\right\}\wedge n^2, & \frac{n\beta^2}{4\sigma^2}\leq 1,
\end{cases}$$
where the bound $\mathbb{E}_r\ell_2(\hat{r},r)\lesssim n^2$ is automatically satisfied by the definition of the loss. When $\frac{n\beta^2}{4\sigma^2}>1$, the result for $\ell_q(\hat{r},r)$ is immediately implied by the relation $\ell_q(\hat{r},r)\leq \ell_2(\hat{r},r)$. When $\frac{n\beta^2}{4\sigma^2}\leq 1$, one can use H\"{o}lder's inequality and Jensen's inequality and get
$$\mathbb{E}_r\ell_q(\hat{r},r)\leq \mathbb{E}_r\left[\ell_2(\hat{r},r)^{q/2}\right]\leq \left[\mathbb{E}_r\ell_2(\hat{r},r)\right]^{q/2}.$$
The proof is complete by taking supreme over $r\in\mathcal{R}$.
\end{proof}

Now we will prove Theorem \ref{thm:upper2}. The following lemma that uniformly controls the comparison between objective functions is important. Its proof will be given in Section \ref{sec:pf-aux}.
\begin{lemma}\label{lem:uniform}
Assume (\ref{eq:error}), (\ref{eq:signal}) and (\ref{eq:signal-entropy}).
For any $r\in\mathcal{R}$, and $t>0$ and any $l\geq 1$, we have
$$\mathbb{P}_{r}\left(\min_{\{\tilde{r}\in\mathcal{R}: tl<\|\tilde{r}-r\|\leq t(l+1)\}} L(\tilde{r})\leq L(r)\right)\lesssim \exp\left(-C'\frac{\beta^2t^2l^2}{M\sigma^2}\right),$$
where $C'>0$ is some universal constant, and $M$ is the same constant in (\ref{eq:signal-entropy}).
\end{lemma}

\begin{proof}[Proof of Theorem \ref{thm:upper2}]
In view of Theorem \ref{thm:upper1}, we only need to improve the polynomial rate in the regime $\frac{n\beta^2}{4\sigma^2}\leq 1$. In other words, it suffices to prove $\mathbb{E}\|\hat{r}-r\|^2\lesssim \frac{\sigma^2}{\beta^2}$.

For a given $t>0$, we will derive a bound for $\mathbb{P}_r(\|\hat{r}-r\|>t)$. Define the set
$$\tilde{\mathcal{R}}_l=\left\{\tilde{r}\in\mathcal{R}: tl<\|\tilde{r}-r\|\leq t(l+1)\right\}.$$
Then, we have
$$
\mathbb{P}_r(\|\hat{r}-r\|>t) \leq \sum_{l=1}^{\infty}\mathbb{P}_r\left(\hat{r}\in\tilde{\mathcal{R}}_l\right) \leq \sum_{l=1}^{\infty}\mathbb{P}_r\left(\inf_{\tilde{r}\in\tilde{\mathcal{R}}_l}L(\tilde{r})\leq L(r)\right).
$$
Using Lemma \ref{lem:uniform}, we have
$$\sum_{l=1}^{\infty}\mathbb{P}_r\left(\inf_{\tilde{r}\in\tilde{\mathcal{R}}_l}L(\tilde{r})\leq L(r)\right)\lesssim \sum_{l=1}^{\infty}\exp\left(-C'\frac{\beta^2t^2l^2}{M\sigma^2}\right)\lesssim \exp\left(-C'\frac{\beta^2t^2}{M\sigma^2}\right).$$
Choosing $t^2=\frac{M\sigma^2 \tau}{\beta^2}$, we obtain the bound
$$\mathbb{P}_r\left(\frac{\|\hat{r}-r\|^2}{(M\sigma^2)/\beta^2}>\tau\right)\lesssim \exp\left(-C'\tau\right).$$
Integrate up the tail over $\tau>0$, we get $\mathbb{E}\|\hat{r}-r\|^2\lesssim \frac{\sigma^2}{\beta^2}$. The desired bound for $\mathbb{E}_r\ell_2(\hat{r},r)$ is implied by $\ell_2(\hat{r},r)=\|\hat{r}-r\|^2/n$ and $\ell_2(\hat{r},r)\lesssim n^2$. For the loss $\ell_q(\hat{r},r)$ with $q\in[0,2)$, we use the inequality $\ell_q(\hat{r},r)\leq\ell_2(\hat{r},r)^{q/2}$, and then get
$$\mathbb{P}_r\left(\frac{\ell_q(\hat{r},r)}{\left(\frac{M\sigma^2}{n\beta^2}\right)^{q/2}}>\tau^{q/2}\right)\lesssim \exp\left(-C'\tau\right).$$
The desired bound for $\mathbb{E}_r\ell_q(\hat{r},r)$ is obtained by integrating up the tail and the fact that $\ell_q(\hat{r},r)\lesssim n^q$. The proof is complete by taking supreme over $r\in\mathcal{R}$.
\end{proof}

\subsection{Proof of Theorem \ref{thm:lower}}

We first show the lower bound in the regime of $\frac{n\beta^2}{4\sigma^2}\leq 1$. We need the following Fano's inequality. The version we give here is from \cite{yu1997assouad}.
\begin{proposition} \label{prop:fano}
Let $(\Xi,\ell)$ be a metric space and $\{\mathbb{P}_{\xi}:\xi\in\Xi\}$ be a collection of probability measures. For any totally bounded $T\subset\Xi$, define the Kullback-Leibler diameter by
$$d_{\text{KL}}(T)=\sup_{\xi,\xi'\in T}D(\mathbb{P}_{\xi}||\mathbb{P}_{\xi'}).$$
Then
$$
\label{eq:fanoKL}\inf_{\hat{\xi}}\sup_{\xi\in\Xi}\mathbb{P}_{\xi}\left\{\ell^2\Big(\hat{\xi}(X),\xi\Big)\geq\frac{\epsilon^2}{4}\right\} \geq 1-\frac{d_{\text{KL}}(T)+\log 2}{\log\mathcal{M}(\epsilon,T,\ell)},
$$
for any $\epsilon>0$, where the packing number $\mathcal{M}(\epsilon,T,\ell)$ is the largest number of points in $T$ that are at least $\epsilon$ away from each other with respect to $\ell(\cdot,\cdot)$.
\end{proposition}
We also need the following Varshamov-Gilbert bound. The version we present here is due to \cite[Lemma 4.7]{massart2007}.
\begin{lemma}\label{lem:VG}
There exists a subset $\{\omega_1,...,\omega_N\}\subset \{0,1\}^d$ such that
\begin{equation}
||\omega_i-\omega_j||^2\geq\frac{d}{4},\quad\text{for any }i\neq j\in[N],\label{eq:defH}
\end{equation}
for some $N\geq \exp {(d/8)}$. 
\end{lemma}

\begin{proof}[Proof of Theorem \ref{thm:lower} when $\frac{n\beta^2}{4\sigma^2}\leq 1$]
We consider the distribution $X_{ij}\sim N(\tilde{\beta}(r(i)-r(j)),\sigma^2)$, which clearly satisfies (\ref{eq:error}). Then, for any two different $r,\tilde{r}\in\mathcal{R}$,
\begin{eqnarray*}
&& \sum_{1\leq i\neq j\leq n}\left(\tilde{\beta}(r(i)-r(j))-\tilde{\beta}(\tilde{r}(i)-\tilde{r}(j))\right)^2 \\
&=& \tilde{\beta}^2\sum_{1\leq i\neq j\leq n}\left((r(i)-\tilde{r}(i))^2+(r(j)-\tilde{r}(j))^2-2(r(i)-\tilde{r}(i))(r(j)-\tilde{r}(j))\right) \\
&=& 2(n-1)\tilde{\beta}^2\|r-\tilde{r}\|^2 + 2\tilde{\beta}^2\|r-\tilde{r}\|^2 -2\tilde{\beta}^2\left(\sum_{i=1}^n(r(i)-\tilde{r}(i))\right)^2 \\
&=& 2n\tilde{\beta}^2\|r-\tilde{r}\|^2 -2\tilde{\beta}^2\left(\sum_{i=1}^n(r(i)-\tilde{r}(i))\right)^2.
\end{eqnarray*}
Note that $\left(\sum_{i=1}^n(r(i)-\tilde{r}(i))\right)^2=O(c_n^2)=o(n)$, and therefore
$$\sum_{1\leq i\neq j\leq n}\left(\tilde{\beta}(r(i)-r(j))-\tilde{\beta}(\tilde{r}(i)-\tilde{r}(j))\right)^2=2n(1+o(1))\tilde{\beta}^2\|r-\tilde{r}\|^2.$$
The conditions (\ref{eq:signal}) and (\ref{eq:signal-entropy}) are satisfied with $\beta=(1+o(1))\tilde{\beta}$.
We first translate the Fano's inequality in Proposition \ref{prop:fano} into the following version,
$$\inf_{\hat{r}}\sup_{r\in\mathcal{R}}\mathbb{P}_r\left(\|\hat{r}-r\|_q\geq \frac{m^{1/q}}{2}\right)\geq 1-\frac{d_{\text{KL}}(T)+\log 2}{\log \mathcal{M}(m^{1/q},T,\|\cdot\|_q)},$$
and we need to construct a subset $T\subset\mathcal{R}$. For simplicity of notation, we assume that $n/6$ is an integer, and otherwise simple modification of the proof can be made. Consider a vector $t\in[n]^n$. We set $t(i)=0$ if $1\leq i<n/3$ or $2n/3<i\leq n$. For the entries between $n/3$ and $2n/3$, each $t(i)$ takes value from $\{\ceil{(m/n)^{1/q}}, 2\ceil{(m/n)^{1/q}}\}$ for $n/3\leq i\leq n/2$, and $t(i)=-t(n-i)$ for all $n/2<i\leq 2n/3$. For any such $t$, we define $r_t\in[n]^n$ by $r_t(i)=i+t(i)$. Since $\sum_{i=1}^n r_t(i) = \sum_{i=1}^n i + \sum_{i=1}^it(i)=\sum_{i=1}^n i$, we have $r_t\in\mathcal{R}$ as long as $m\leq n^{q+1}/32$. Moreover, for any two $t,t'$, we have
$$\sum_{i=1}^n|r_t(i)-r_{t'}(i)|^q=2\sum_{i=n/3}^{n/2}|t(i)-t'(i)|^q.$$
Therefore, by Lemma \ref{lem:VG}, there exists a set $T\subset\mathcal{R}$ that collects those $r_t$'s, such that for any two different $r_t,r_{t'}\in T$, $n\ell_q(r_t,r_{t'})\geq c_1m$ and $|T|\geq e^{c_2n}$. We bound the Kullback-Leiber diameter by
\begin{eqnarray*}
d_{\text{KL}}(T) &=& \sup_{r_t,r_{t'}\in T}\sum_{1\leq i\neq j\leq n}\left(\mu_{r_t(i)r_t(j)}-\mu_{r_{t'}(i)r_{t'}(j)}\right)^2/(2\sigma^2) \\
&\leq& \sup_{r_t,r_{t'}\in T}\frac{n\tilde{\beta}^2\|r_t-r_{t'}\|^2}{\sigma^2} \leq C\frac{n^{2-2/q}\tilde{\beta}^2m^{2/q}}{\sigma^2}.
\end{eqnarray*}
The Fano's inequality then implies
$$\inf_{\hat{r}}\sup_{r\in\mathcal{R}}\mathbb{P}_r\left(n\ell_q(\hat{r},r)\geq\frac{c_1m}{2^q}\right)\geq 1-\frac{C\frac{n^{2-2/q}\tilde{\beta}^2m^{2/q}}{\sigma^2}+\log 2}{c_2n}.$$
We take $m=c\left(n\left(\frac{\sigma^2}{n\beta^2}\right)^{q/2}\wedge n^{q+1}\right)$ for some sufficiently small constant $c>0$, and then we get the desired lower bound for $\inf_{\hat{r}}\sup_{r\in\mathcal{R}}\mathbb{E}_r\ell_2(\hat{r},r)$ by applying a Markov inequality.
\end{proof}

\begin{proof}[Proof of Theorem \ref{thm:lower} when $\frac{n\beta^2}{4\sigma^2}> 1$]
We consider the distribution $X_{ij}\sim N(\tilde{\beta}(r(i)-r(j)),\sigma^2)$ as in the previous part of the proof. The $\tilde{\beta}$ is chosen as $(1+o(1))\beta$ so that (\ref{eq:error}), (\ref{eq:signal}) and (\ref{eq:signal-entropy}) are satisfied.
Since
$$\inf_{\hat{r}}\sup_{r\in\mathcal{R}}\mathbb{E}_r\ell_q(\hat{r},r)\geq \inf_{\hat{r}}\sup_{r\in\mathcal{R}}\mathbb{E}_r\ell_0(\hat{r},r),$$
we only need to prove the lower bound for $\inf_{\hat{r}}\sup_{r\in\mathcal{R}}\mathbb{E}_r\ell_0(\hat{r},r)$. Define
$$\tilde{\mathcal{R}}=\left\{r\in\mathcal{R}:\left|\sum_{i=1}^nr(i)-\sum_{i=1}^n i\right|\leq 1\right\}.$$
Then, we have
\begin{equation}
\inf_{\hat{r}}\sup_{r\in\mathcal{R}}\mathbb{E}_rn\ell_0(\hat{r},r)\geq \inf_{\hat{r}}\sup_{r\in\tilde{\mathcal{R}}}\mathbb{E}_rn\ell_0(\hat{r},r)\geq \inf_{\hat{r}}\sum_{i=1}^n\frac{1}{|\tilde{\mathcal{R}}|}\sum_{r\in\tilde{\mathcal{R}}}\mathbb{P}_r\left(\hat{r}(i)\neq r(i)\right).\label{eq:B-r}
\end{equation}
For an $i\in\{2,3,...,n-1\}$, any $r\in\tilde{\mathcal{R}}$ can be written as $r=(r_i,r_{-i})$ with some slight abuse of notation, where we use $r_i$ to denote the $i$th entry of $r$ and $r_{-i}$ to denote the remaining entries. Then, the set $\tilde{R}$ has the following decomposition
$$\tilde{\mathcal{R}}=\cup_{r_{-i}}\mathcal{R}_{r_{-i}},$$
where all the elements in $\mathcal{R}_{r_{-i}}$ have the same entries except for the $i$th one. It is easy to see that 
$$|\tilde{\mathcal{R}}|=\sum_{r_{-i}}|\mathcal{R}_{r_{-i}}|.$$
According to the definition of $\tilde{\mathcal{R}}$, for any $r\in\tilde{\mathcal{R}}$, the sum $\sum_{i=1}^nr(i)$ only takes three possible values in $\left\{\left(\sum_{i=1}^n i\right)-1,\sum_{i=1}^n i,\left(\sum_{i=1}^n i\right)+1\right\}$. Therefore, for each possible $r_{-i}$ with $i\in\{2,3,...,n-1\}$, we have $|\mathcal{R}_{r_{-i}}|=3$. We then take a subset $\tilde{\mathcal{R}}_{r_{-i}}\subset\mathcal{R}_{r_{-i}}$ with $|\tilde{\mathcal{R}}_{r_{-i}}|=2$ so that the two elements in $\tilde{\mathcal{R}}_{r_{-i}}$ satisfy $\|r-r'\|^2=1$. We continue to lower bound (\ref{eq:B-r}) by
\begin{eqnarray*}
&& \inf_{\hat{r}}\sum_{i=2}^{n-1}\frac{1}{|\tilde{\mathcal{R}}|}\sum_{r\in\tilde{\mathcal{R}}}\mathbb{P}_r\left(\hat{r}(i)\neq r(i)\right) \\
&\geq& \inf_{\hat{r}}\sum_{i=2}^{n-1}\frac{1}{|\mathcal{R}|}\sum_{r_{-i}}|\mathcal{R}_{r_{-i}}|\frac{1}{|\mathcal{R}_{r_{-i}}|}\sum_{r\in\mathcal{R}_{r_{-i}}}\mathbb{P}_r(\hat{r}(i)\neq r(i)) \\
&\geq& \inf_{\hat{r}}\frac{1}{3}\sum_{i=2}^{n-1}\frac{1}{|\mathcal{R}|}\sum_{r_{-i}}|\mathcal{R}_{r_{-i}}|\sum_{r\in\tilde{\mathcal{R}}_{r_{-i}}}\mathbb{P}_r(\hat{r}(i)\neq r(i)) \\
&\geq& \frac{1}{3}\sum_{i=2}^{n-1}\frac{1}{|\mathcal{R}|}\sum_{r_{-i}}|\mathcal{R}_{r_{-i}}|\inf_{\hat{r}(i)}\sum_{r\in\tilde{\mathcal{R}}_{r_{-i}}}\mathbb{P}_r(\hat{r}(i)\neq r(i)) \\
&=& \frac{1}{3}\sum_{i=2}^{n-1}\frac{1}{|\mathcal{R}|}\sum_{r_{-i}}|\mathcal{R}_{r_{-i}}|\inf_{\hat{r}}\sum_{r\in\tilde{\mathcal{R}}_{r_{-i}}}\mathbb{P}_r(\hat{r}\neq r).
\end{eqnarray*}
Now it is sufficient to lower bound each $\inf_{\hat{r}}\sum_{r\in\tilde{\mathcal{R}}_{r_{-i}}}\mathbb{P}_r(\hat{r}\neq r)$. We denote the two elements in $\tilde{\mathcal{R}}_{r_{-i}}$ by $r$ and $\tilde{r}$. Then, by Neyman-Pearson lemma, we have
$$\inf_{\hat{r}}\sum_{r\in\tilde{\mathcal{R}}_{r_{-i}}}\mathbb{P}_r(\hat{r}\neq r)=\mathbb{P}_r(L(\tilde{r})\leq L(r))+\mathbb{P}_{\tilde{r}}(L(r)\leq L(\tilde{r}))=2\mathbb{P}\left(N(0,1)>\frac{\sqrt{n-1}\tilde{\beta}}{\sqrt{2}\sigma}\right),$$
where the last identity above is using the fact that $\|r-\tilde{r}\|^2=1$. By a standard normal tail bound argument, we have
$$\mathbb{P}\left(N(0,1)>\frac{\sqrt{n-1}\tilde{\beta}}{\sqrt{2}\sigma}\right)\gtrsim \exp\left(-\frac{(n-1)\tilde{\beta}^2}{4\sigma^2}\right)=\exp\left(-(1+o(1))\frac{n\beta^2}{4\sigma^2}\right).$$
Therefore,
\begin{eqnarray*}
\inf_{\hat{r}}\sup_{r\in\mathcal{R}}\mathbb{E}_r\ell_0(\hat{r},r) &\geq& \frac{1}{3n}\sum_{i=2}^{n-1}\frac{1}{|\mathcal{R}|}\sum_{r_{-i}}|\mathcal{R}_{r_{-i}}|\inf_{\hat{r}}\sum_{r\in\tilde{\mathcal{R}}_{r_{-i}}}\mathbb{P}_r(\hat{r}\neq r) \\
&\gtrsim& \exp\left(-(1+o(1))\frac{n\beta^2}{4\sigma^2}\right),
\end{eqnarray*}
and the proof is complete.
\end{proof}

\subsection{Proofs of Theorem \ref{thm:exact} and Theorem \ref{thm:exact-lower}}

The proof of Theorem \ref{thm:exact} is a simple application of Markov inequality.
\begin{proof}[Proof of Theorem \ref{thm:exact}]
Under the assumption, there exists a small positive constant $\delta>0$, such that
$$\frac{n\beta^2}{4\sigma^2}\geq (1+\delta)\log n,$$
for any sufficiently large $n$. Then, we have
\begin{eqnarray*}
\mathbb{P}_r\left(\hat{r}\neq r\right) &=& \mathbb{P}_r\left(\ell_0(\hat{r},r)\geq \frac{1}{n}\right) \\
&\leq& n\mathbb{E}_r\ell_0(\hat{r},r) \\
&\leq& n\exp\left(-(1+o(1))\frac{n\beta^2}{4\sigma^2}\right) \\
&\leq& n^{-(1+o(1))\delta}.
\end{eqnarray*}
The proof is complete by letting $n$ tend to infinity.
\end{proof}

To prove Theorem \ref{thm:exact-lower}, we needs the following bound for the maximum of dependent Gaussian random variables, which is a result in \cite{hartigan2013bounding}.
\begin{lemma}\label{lem:hartigan}
Consider zero-mean Gaussian random variables $W_1,...,W_n$ with covariance matrix of minimum eigenvalue $\lambda$ and maximum eigenvalue $\Lambda$. Then, for $n\geq 70$,
$$\mathbb{P}\left(\max_{1\leq i\leq n}W_i\geq \lambda\left(2\log n-2.5-\log(2\log n-2.5)\right)^{1/2}-0.68\Lambda\right)\geq \frac{1}{2}.$$
\end{lemma}
\begin{proof}[Proof of Theorem \ref{thm:exact-lower}]
We again consider the distribution $X_{ij}\sim N(\tilde{\beta}(r(i)-r(j)),\sigma^2)$ as in the proof of Theorem \ref{thm:lower}. The $\tilde{\beta}$ is chosen as $(1+o(1))\beta$ so that (\ref{eq:error}), (\ref{eq:signal}) and (\ref{eq:signal-entropy}) are satisfied. In the proof, we assume that $n$ is an even number, otherwise a simple modification of the argument can be done. Note that
$$\inf_{\hat{r}}\sup_{r\in\mathcal{R}}\mathbb{P}_r(\hat{r}\neq r)\geq \inf_{\hat{r}}\frac{1}{|\mathcal{R}|}\sum_{r\in\mathcal{R}}\mathbb{P}_r(\hat{r}\neq r)=\frac{1}{|\mathcal{R}|}\sum_{r\in\mathcal{R}}\mathbb{P}_r(\hat{r}\neq r),$$
where $\hat{r}=\argmin_{r\in\mathcal{R}}L(r)$ is the MLE. Therefore, it suffices to give a lower bound for each $\mathbb{P}_r(\hat{r}\neq r)$. Without loss of generality, we consider $r$ with $r(i)=i$. Then, for each $j\in[n]$, we choose an $r_j\in\mathcal{R}$ that differs with $r$ only at the $j$th position and $\|r_j-r\|^2=1$. We have the lower bound
$$\mathbb{P}_r(\hat{r}\neq r)\geq \mathbb{P}_r\left(\min_{j\in[n]}L(r_j)<L(r)\right).$$
After some rearrangement, we have
$$\mathbb{P}_r\left(\min_{j\in[n]}L(r_j)<L(r)\right)=\mathbb{P}_r\left(\min_{j\in[n]}\sum_{i\in[n]\backslash\{j\}}(Z_{ij}-Z_{ji})\leq -(n-1)\beta/\sigma\right),$$
where $Z_{ij}=X_{ij}-\tilde{\beta}(i-j)\sim N(0,\sigma^2)$. The above probability equals
$$\mathbb{P}\left(\max_{j\in[n]}W_j>\frac{\sqrt{n-1}\tilde{\beta}}{\sqrt{2}\sigma}\right),$$
where $W_j\sim N(0,1)$ and $\mathbb{E}W_iW_j=(n-1)^{-1}$ for all $i\neq j$. Using a similar argument, we can show the same lower bound for $\mathbb{P}_r(\hat{r}\neq r)$ with any other $r\in\mathcal{R}$. Therefore,
$$\frac{1}{|\mathcal{R}|}\sum_{r\in\mathcal{R}}\mathbb{P}_r(\hat{r}\neq r)\geq \mathbb{P}\left(\max_{j\in[n]}W_j>\frac{\sqrt{n-1}\tilde{\beta}}{\sqrt{2}\sigma}\right).$$
Under the assumption that $\limsup_n\frac{n\beta^2}{4\sigma^2\log n}<1$, there exists a small $\delta>0$, such that $\frac{n\beta^2}{4\sigma^2}<(1-\delta)\log n$ for any sufficiently large $n$. To use Lemma \ref{lem:hartigan}, it is easy to check that the smallest and the largest eigenvalues of the covariance matrix of $W_1,...,W_n$ are $\frac{n-2}{n-1}$ and $2$, respectively. Therefore, by Lemma \ref{lem:hartigan}, we have
$$\mathbb{P}\left(\max_{j\in[n]}W_j>\frac{\sqrt{n-1}\tilde{\beta}}{\sqrt{2}\sigma}\right)\geq \frac{1}{2},$$
for any sufficiently large $n$. Hence, the proof is complete.
\end{proof}

\subsection{Proofs of Theorem \ref{thm:upper-comp}, Theorem \ref{thm:upper-add} and Theorem \ref{thm:PL}}

Before the proof of each theorem, we state one or two auxiliary lemmas. The proofs of these lemmas will be given in Section \ref{sec:pf-aux}.
\begin{lemma}\label{lem:dev-comp}
Consider the differential comparison model that satisfies (\ref{eq:error}), (\ref{eq:signal-comp}) and (\ref{eq:stable-comp}).
For any $m$ such that $\mathcal{R}_m\neq\varnothing$, we have
$$\max_{\tilde{r}\in\mathcal{R}_m}\mathbb{P}_r\left(\sum_{i=1}^n(S_i-\theta_{\tilde{r}(i)})^2\leq \sum_{i=1}^n(S_i-\theta_{r(i)})^2\right)\leq\exp\left(-(1+o(1))\frac{n\beta^2m}{4\sigma^2}\right),$$
where $S_i$ and $\mathcal{R}_m$ are defined in (\ref{eq:score-comp}) and (\ref{eq:def-rm}).
\end{lemma}
\begin{proof}[Proof of Theorem \ref{thm:upper-comp}]
For the exponential rate in the regime $\frac{n\beta^2}{4\sigma^2}>1$, the proof is the same as that of Theorem \ref{thm:upper1}. One only needs to replace Lemma \ref{lem:dev} by Lemma \ref{lem:dev-comp}. Now we give the proof of the polynomial rate in the regime $\frac{n\beta^2}{4\sigma^2}\leq 1$. We use the notation
$\Delta_r=\sum_{j=1}^n\theta_{r(j)}-\sum_{j=1}^n\theta_j$.
By (\ref{eq:stable-comp}), we have $\max_{r\in\mathcal{R}}|\Delta_r|=o(\sqrt{n}\beta)$. Note that for each $i$, $\mathbb{E}_rS_i=\theta_{r(i)}-\Delta_r/n$. Using the condition (\ref{eq:signal-comp}), we have
\begin{eqnarray}
\nonumber \|\hat{r}-r\|^2 &\leq& \frac{1}{\beta^2}\sum_{i=1}^n(\theta_{\hat{r}(i)}-\theta_{r(i)})^2 \\
\nonumber &\leq& \frac{2}{\beta^2}\sum_{i=1}^n(\theta_{\hat{r}(i)}-S_i)^2 + \frac{2}{\beta^2}\sum_{i=1}^n(S_i-\theta_{r(i)})^2 \\
\label{eq:def-fm} &\leq& \frac{4}{\beta^2}\sum_{i=1}^n(S_i-\theta_{r(i)})^2 \\
\label{eq:pu} &\leq& \frac{8}{\beta^2}\sum_{i=1}^n(S_i-\theta_{r(i)}-\Delta_r/n)^2 + \frac{8\Delta_r^2}{n\beta^2},
\end{eqnarray}
where the inequality (\ref{eq:def-fm}) is by the definition of $\hat{r}$ in (\ref{eq:feature-matching}). We use the notation $Z_{ij}=X_{ij}-\mu_{r(i)r(j)}$. For each $i\in[n]$, we have
$$\mathbb{E}(S_i-\theta_{r(i)}-\Delta_r/n)^2=\mathbb{E}\left(\frac{1}{2n}\sum_{j\in[n]\backslash\{i\}}(Z_{ij}-Z_{ji})\right)^2\lesssim \frac{\sigma^2}{n},$$
by using the condition (\ref{eq:error}). Therefore, the bound (\ref{eq:pu}) implies
$$\mathbb{E}_r\ell_2(\hat{r},r)\lesssim \frac{\sigma^2}{n\beta^2}+\frac{\Delta_r^2}{n^2\beta^2},$$
where the second term $\frac{\Delta_r^2}{n^2\beta^2}$ is negligible given the condition (\ref{eq:stable-comp}) and $\frac{n\beta^2}{4\sigma^2}\leq 1$. Hence, we obtain the bound $\mathbb{E}_r\ell_2(\hat{r},r)\lesssim \frac{\sigma^2}{n\beta^2}\wedge n^2$. The bound for $\mathbb{E}_r\ell_q(\hat{r},r)$ is immediately implied by $\mathbb{E}_r\ell_q(\hat{r},r)\leq \mathbb{E}_r\left[\ell_2(\hat{r},r)^{q/2}\right]\leq \left[\mathbb{E}_r\ell_2(\hat{r},r)\right]^{q/2}$. The exact recovery result follows a simple application of Markov inequality as is done in the proof of Theorem \ref{thm:exact}. Thus, the proof is complete.
\end{proof}

\begin{lemma}\label{lem:dev-add}
Consider the additive collaboration model that satisfies (\ref{eq:error}) and (\ref{eq:signal-comp}).
For any $m$ such that $\mathcal{R}_m\neq\varnothing$, we have
$$\max_{\tilde{r}\in\mathcal{R}_m}\mathbb{P}_r\left(\sum_{i=1}^n(S_i-\theta_{\tilde{r}(i)})^2\leq \sum_{i=1}^n(S_i-\theta_{r(i)})^2\right)\leq\exp\left(-(1+o(1))\frac{n\beta^2m}{4\sigma^2}\right),$$
where $S_i$ and $\mathcal{R}_m$ are defined in (\ref{eq:score-coll}) and (\ref{eq:def-rm}).
\end{lemma}

\begin{proof}[Proof of Theorem \ref{thm:upper-add}]
For the exponential rate in the regime $\frac{n\beta^2}{4\sigma^2}>1$, the proof is the same as that of Theorem \ref{thm:upper1}. One only needs to replace Lemma \ref{lem:dev} by Lemma \ref{lem:dev-add}. For the polynomial rate in the regime $\frac{n\beta^2}{4\sigma^2}\leq 1$, we use the same argument that leads to (\ref{eq:def-fm}), and we have
$$\|\hat{r}-r\|^2\leq \frac{4}{\beta^2}\sum_{i=1}^n(S_i-\theta_{r(i)})^2.$$
By the definition (\ref{eq:score-coll}) and the condition (\ref{eq:error}),
$$\mathbb{E}_r(S_i-\theta_{r(i)})^2=\mathbb{E}_r\left(\frac{1}{2(n-2)}\left(\sum_{j\in[n]\backslash\{i\}}(Z_{ij}+Z_{ji})-\frac{1}{n-1}\sum_{1\leq i\neq j\leq n}Z_{ij}\right)\right)^2\lesssim \frac{\sigma^2}{n}.$$
This implies $\mathbb{E}_r\ell_2(\hat{r},r)\lesssim \frac{\sigma^2}{n\beta^2}\wedge n^2$ and $\mathbb{E}_r\ell_q(\hat{r},r)\leq \mathbb{E}_r\left[\ell_2(\hat{r},r)^{q/2}\right]\leq \left[\mathbb{E}_r\ell_2(\hat{r},r)\right]^{q/2}\lesssim \left(\frac{\sigma^2}{n\beta^2}\wedge n^2\right)^{q/2}$. The exact recovery result follows a simple application of Markov inequality as is done in the proof of Theorem \ref{thm:exact}. Thus, the proof is complete.
\end{proof}

\begin{lemma}\label{lem:dev-PL}
Consider either the differential comparison model or the additive collaboration model with $\theta_i=\alpha+\tilde{\beta}i$ for some $\tilde{\beta}=(1+o(1))\beta$. Assume the condition (\ref{eq:error}) and $\frac{n\beta^2}{4\sigma^2}>1$. For any $m$ such that $\mathcal{R}_m\neq\varnothing$, we have
$$\max_{\tilde{r}\in\mathcal{R}_m'}\mathbb{P}_r\left(\text{PL}(\tilde{r})\leq \text{PL}(r)\right)\leq 2\exp\left(-(1+o(1))\left(1-C\sqrt{m/n^3}\right)\frac{n\beta^2m}{4\sigma^2}\right),$$
for some constant $C>0$,
where $\mathcal{R}_m'=\mathcal{R}_m\cap\mathcal{R}'$ and $S_i$ is defined in Section \ref{sec:para}.
\end{lemma}
\begin{lemma}\label{lem:uniform-PL}
Consider either the differential comparison model or the additive collaboration model with $\theta_i=\alpha+\tilde{\beta}i$ for some $\tilde{\beta}=(1+o(1))\beta$. Assume the condition (\ref{eq:error}).
For any $r\in\mathcal{R}'$, and $t>0$ and any $l\geq 1$, we have
$$\mathbb{P}_{r}\left(\min_{\{\tilde{r}\in\mathcal{R}': tl<\|\tilde{r}-r\|\leq t(l+1)\}} \text{PL}(\tilde{r})\leq \text{PL}(r)\right)\leq 3\exp\left(-n\left(\frac{C_1\beta^2t^2l^2}{\sigma^2}-C_2\right)\right),$$
where $C_1,C_2>0$ are some universal constants.
\end{lemma}
\begin{proof}[Proof of Theorem \ref{thm:PL}]
The exponential rate is by the same proof of Theorem \ref{thm:upper1} with the help of Lemma \ref{lem:card} and Lemma \ref{lem:dev-PL}. For the polynomial rate, we define the set
$$\tilde{\mathcal{R}}_l'=\left\{\tilde{r}\in\mathcal{R}: tl<\|\tilde{r}-r\|\leq t(l+1)\right\}.$$
Then, we have $
\mathbb{P}_r(\|\hat{r}-r\|>t) \leq \sum_{l=1}^{\infty}\mathbb{P}_r\left(\hat{r}\in\tilde{\mathcal{R}}'_l\right) \leq \sum_{l=1}^{\infty}\mathbb{P}_r\left(\inf_{\tilde{r}\in\tilde{\mathcal{R}}_l'}\text{PL}(\tilde{r})\leq \text{PL}(r)\right)$. By Lemma \ref{lem:uniform-PL}, we have the bound
$$3\sum_{l=1}^{\infty}\exp\left(-n\left(\frac{C_1\beta^2t^2l^2}{\sigma^2}-C_2\right)\right)\lesssim \exp\left(-n\left(\frac{C_1\beta^2t^2}{\sigma^2}-C_2\right)\right).$$
Therefore, there are some constants $C,C'>0$, such that
$$\mathbb{P}_r\left(\|\hat{r}-r\|^2> C\frac{\sigma^2}{\beta^2}(1+x)\right)\lesssim \exp\left(-C'nx\right),$$
for any $x>0$.
Integrating up the tail, we obtain the desired polynomial convergence rate for $\mathbb{E}_r\ell_2(\hat{r},r)$. The result for $\mathbb{E}_r\ell_q(\hat{r},r)$ is by the inequality $\mathbb{E}_r\ell_q(\hat{r},r)\leq \mathbb{E}_r\left[\ell_2(\hat{r},r)^{q/2}\right]\leq \left[\mathbb{E}_r\ell_2(\hat{r},r)\right]^{q/2}$. The exact recovery result follows a simple application of Markov inequality as is done in the proof of Theorem \ref{thm:exact}. The proof is complete.
\end{proof}

\subsection{Proofs of Theorem \ref{thm:upper-Poi}, Theorem \ref{thm:lower-Poi} and Theorem \ref{thm:exact-Poi}}\label{sec:pf-thm-last}

We first give a lemma to facilitate the proof of Theorem \ref{thm:upper-Poi}.
\begin{lemma}\label{lem:dev-Poi}
Assume (\ref{eq:signal-Poi}).
For any $m$ such that $\mathcal{R}_m\neq \varnothing$ that is defined in (\ref{eq:def-rm}), we have
$$\max_{\tilde{r}\in\mathcal{R}_m}\mathbb{P}_r\left(\prod_{1\leq i\neq j\leq n}\frac{\mu_{\tilde{r}(i)\tilde{r}(j)}^{X_{ij}}e^{-\mu_{\tilde{r}(i)\tilde{r}(j)}}}{X_{ij}!}\geq \prod_{1\leq i\neq j\leq n}\frac{\mu_{r(i)r(j)}^{X_{ij}}e^{-\mu_{r(i)r(j)}}}{X_{ij}!}\right)\leq \exp\left(-n\beta^2m\right).$$
\end{lemma}
\begin{proof}[Proof of Theorem \ref{thm:upper-Poi}]
The proof is essentially the same as that of Theorem \ref{thm:upper1}. The only difference is that we use Lemma \ref{lem:dev-Poi} instead of Lemma \ref{lem:dev}. Therefore, we only need to replace all the $\frac{n\beta^2}{4\sigma^2}$ in the proof of Theorem \ref{thm:upper1} by $n\beta^2$ and obtain the desired result.
\end{proof}

\begin{proof}[Proof of Theorem \ref{thm:lower-Poi}]
We consider the distribution $X_{ij}\sim\text{Poisson}(\mu_{r(i)r(j)})$, where $\sqrt{\mu_{ij}}=2\alpha+\tilde{\beta}(i+j)$. We set $\alpha=\tilde{\beta}n^2$. Note that for any $r,\tilde{r}\in\mathcal{R}$,
\begin{eqnarray*}
&& \sum_{1\leq i\neq j\leq n}\left(\sqrt{\mu_{\tilde{r}(i)\tilde{r}(j)}}-\sqrt{\mu_{r(i)r(j)}}\right)^2 \\
&=& \tilde{\beta}^2\sum_{1\leq i\neq j\leq n}\left(\tilde{r}(i)+\tilde{r}(j)-r(i)-r(j)\right)^2 \\
&=& \tilde{\beta}^2\left(2(n-2)\|\tilde{r}-r\|^2+2\left(\sum_{i=1}^n(r(i)-\tilde{r}(i))\right)^2\right) \\
&=& 2(1+o(1))n\tilde{\beta}^2\|\tilde{r}-r\|^2.
\end{eqnarray*}
Therefore, the condition (\ref{eq:signal-Poi}) is satisfied with some $\beta=(1+o(1))\tilde{\beta}$. We first derive polynomial lower bounds when $n\beta^2\leq 1$. We use the same argument in the proof of Theorem \ref{thm:lower}. The only difference is the calculation of $d_{\text{KL}}(T)$. For any $r_t,r_{t'}\in T\subset\mathcal{R}$,
\begin{eqnarray}
\nonumber && D\left(\otimes_{1\leq i\neq j\leq n}\text{Poisson}(\mu_{r(i)r(j)})\|\otimes_{1\leq i\neq j\leq n}\text{Poisson}(\mu_{\tilde{r}(i)\tilde{r}(j)})\right) \\
\nonumber &=& \sum_{1\leq i\neq j\leq n}\left(\mu_{r(i)r(j)}\log\frac{\mu_{r(i)r(j)}}{\mu_{\tilde{r}(i)\tilde{r}(j)}}-\mu_{r(i)r(j)}+\mu_{\tilde{r}(i)\tilde{r}(j)}\right) \\
\nonumber &\leq& \sum_{1\leq i\neq j\leq n}\frac{(\mu_{r(i)r(j)}-\mu_{\tilde{r}(i)\tilde{r}(j)})^2}{\mu_{\tilde{r}(i)\tilde{r}(j)}} \\
\nonumber &=& 4(1+o(1))\sum_{1\leq i\neq j\leq n}\frac{(\mu_{r(i)r(j)}-\mu_{\tilde{r}(i)\tilde{r}(j)})^2}{(\sqrt{\mu_{r(i)r(j)}}+\sqrt{\mu_{\tilde{r}(i)\tilde{r}(j)}})^2},
\end{eqnarray}
where the last inequality is because $\max_{i,j,i',j'}\sqrt{\frac{\mu_{ij}}{\mu_{i'j'}}}= 1+o(1)$. This leads to the bound
$$d_{\text{KL}}(T)\lesssim \max_{r_t,r_{t'}\in T}n\tilde{\beta}^2\|r_t-r_{t'}\|^2.$$
Therefore, the same argument in the proof of Theorem \ref{thm:lower} will go through if we replace every $\frac{n\tilde{\beta}^2}{4\sigma^2}$ by $n\tilde{\beta}^2$. This leads to the desired lower bound $\inf_{\hat{r}}\sup_{r\in\mathcal{R}}\ell_q(\hat{r},r)\gtrsim (n\beta^2)^{-q/2}\wedge n^q$. 

Now we give exponential lower bounds when $n\beta^2> 1$. By the same argument in the proof of Theorem \ref{thm:lower}, the lower bounds are determined by the following quantity
\begin{equation}
\mathbb{P}_r\left(\prod_{1\leq i\neq j\leq n}\frac{\mu_{\tilde{r}(i)\tilde{r}(j)}^{X_{ij}}e^{-\mu_{\tilde{r}(i)\tilde{r}(j)}}}{X_{ij}!}> \prod_{1\leq i\neq j\leq n}\frac{\mu_{r(i)r(j)}^{X_{ij}}e^{-\mu_{r(i)r(j)}}}{X_{ij}!}\right),\label{eq:toL-Poi}
\end{equation}
for some $r$ and $\tilde{r}$ that satisfy $\|\tilde{r}-r\|^2=1$. Without loss of generality, We assume $|\tilde{r}(1)-r(1)|=1$ and $\tilde{r}(i)=r(i)$ for all $i>1$. Then, the above probability can be written as
$$\mathbb{P}_r\left(\prod_{i=2}^n p\left(X_{1i}|\mu_{\tilde{r}(1)i}\right)p\left(X_{i1}|\mu_{i\tilde{r}(1)}\right)>\prod_{i=2}^n p\left(X_{1i}|\mu_{{r}(1)i}\right)p\left(X_{i1}|\mu_{i{r}(1)}\right)\right),$$
where we use $p(X|\mu)$ to denote the probability mass function of $\text{Poisson}(\mu)$. For any $t\in(0,1)$, we use the following general argument
\begin{eqnarray*}
P\left(\frac{q}{p}(X)>1\right) &\geq& P\left(0<t\log\frac{q}{p}(X)<L\right) \\
&=& \int_{\{0<t\log\frac{q}{p}(x)<L\}} p(x) \\
&=& \int_{\{0<t\log\frac{q}{p}(x)<L\}} \frac{p(x)\frac{q(x)^t}{p(x)^t}\int p^{1-t}q^t}{\frac{q(x)^t}{p(x)^t}\int p^{1-t}q^t} \\
&\geq& \int p^{1-t}q^t e^{-L}\int p_t(x) \\
&=& \exp\left(\log \int p^{1-t}q^t -L\right)P_t\left(0<t\log\frac{q}{p}(X)<L\right),
\end{eqnarray*}
where the distribution $P_t$ has density function proportional to $p^{1-t}q^t$. In general, we choose $t$ to minimize $\log \int p^{1-t}q^t$ and $L=t\sqrt{P_t\left(\log\frac{q}{p}\right)^2}$. For $P=\otimes_i\text{Poisson}(\mu_i)$ and $Q=\otimes_i\text{Poisson}(\tilde{\mu}_i)$, direct calculation gives
$$-\log\int p^{1-t}q^t=\sum_i\left(t\tilde{\mu}_i+(1-t)\mu_i-\tilde{\mu}_i^t\mu_i^{1-t}\right).$$
Moreover, $P_t=\otimes_i\text{Poisson}(\mu_i^{1-t}\tilde{\mu}_i^t)$. Under the current setting, $\max_i\frac{|\sqrt{\mu_i}-\sqrt{\tilde{\mu}_i}|}{\min(\sqrt{\mu_i},\sqrt{\tilde{\mu}_i})}=o(1)$. Therefore, a Taylor expansion argument gives the bound
$$\sum_i\left(t\tilde{\mu}_i+(1-t)\mu_i-\tilde{\mu}_i^t\mu_i^{1-t}\right)\leq 2(1+o(1))t(1-t)\sum_i\left(\sqrt{\mu_i}-\sqrt{\tilde{\mu}_i}\right)^2.$$
This leads to
\begin{eqnarray*}
&& \max_{t\in(0,1)}\sum_i\left(t\tilde{\mu}_i+(1-t)\mu_i-\tilde{\mu}_i^t\mu_i^{1-t}\right) \\
&\leq& \frac{1+o(1)}{2}\sum_i\left(\sqrt{\mu_i}-\sqrt{\tilde{\mu}_i}\right)^2 \\
&=&  \frac{1+o(1)}{2}\sum_{i=2}^n\left(\sqrt{\mu_{r(1)i}}-\sqrt{\mu_{\tilde{r}(1)i}}\right)^2 + \frac{1+o(1)}{2}\sum_{i=2}^n\left(\sqrt{\mu_{ir(1)}}-\sqrt{\mu_{i\tilde{r}(1)}}\right)^2 \\
&=& (1+o(1))n\beta^2.
\end{eqnarray*}
On the other hand, it is easy to see that lower bound
$$\max_{t\in(0,1)}\sum_i\left(t\tilde{\mu}_i+(1-t)\mu_i-\tilde{\mu}_i^t\mu_i^{1-t}\right)\geq \frac{1}{2}\sum_i\left(\sqrt{\mu_i}-\sqrt{\tilde{\mu}_i}\right)^2.$$
Therefore, the $t$ that minimizes $\log \int p^{1-t}q^t$ satisfies $t=\frac{1}{2}(1+o(1))$.
We give a bound for $L$. Since $t$ minimizes $\log \int p^{1-t}q^t$, we have $P_t\left(\log\frac{q}{p}\right)^2=\Var_{P_t}\left(\log\frac{q}{p}\right)$. Note that
$$\Var_{P_t}\left(\log\frac{q}{p}\right) = \sum_i \left(\log\frac{\mu_i}{\tilde{\mu}_i}\right)^2\mu_i^{1-t}\tilde{\mu}_i^t\lesssim \sum_i\left(\sqrt{\mu_i}-\sqrt{\tilde{\mu}_i}\right)^2\lesssim n\beta^2,$$
where we have used the fact $\max_i\sqrt{\frac{\mu_i}{\tilde{\mu}_i}}=1+o(1)$. Therefore, we have
$$\exp\left(\log \int p^{1-t}q^t -L\right)\geq \exp\left(-(1+o(1))n\beta^2-\sqrt{Cn\beta^2}\right),$$
for some constant $C>0$. Next, we need to lower bound $P_t\left(0<\frac{t}{L}\log\frac{q}{p}(X)<1\right)$. Since $\frac{1}{L}\log\frac{q}{p}(X)$ is sum of independent random variables, and it is properly standardized, we only need to establish a central limit theorem for $\frac{1}{L}\log\frac{q}{p}(X)$. Thus, it is sufficient to check Lyapunov's condition. Under the current setting, Lyapunov's condition is easily satisfied by $\max_i\sqrt{\frac{\mu_i}{\tilde{\mu}_i}}\leq 1+o(1)$. Hence, $P_t\left(0<\frac{t}{L}\log\frac{q}{p}(X)<1\right)\geq c$ for some constant $c>0$. Therefore, up to a constant, we have obtained the following lower bound for (\ref{eq:toL-Poi}),
$$\exp\left(-(1+o(1))n\beta^2-\sqrt{Cn\beta^2}\right)=\exp\left(-(1+o(1))\left(1+O\left(\frac{1}{\sqrt{n\beta^2}}\right)\right)n\beta^2\right).$$
For the case $n\beta^2=O(1)$, this is of a constant order, which is of the same order as $\exp\left(-(1+o(1))n\beta^2\right)$. For the case $n\beta^2\rightarrow\infty$, $\sqrt{Cn\beta^2}$ is of a smaller order compared with $n\beta^2$, and the bound $\exp\left(-(1+o(1))n\beta^2\right)$ still holds. The proof is complete.
\end{proof}

\begin{proof}[Proof of Theorem \ref{thm:exact-Poi}]
The upper bound follows the simple argument of Markov inequality as is done in the proof of Theorem \ref{thm:exact}. We give the proof of the lower bound by following the same argument in the proof of Theorem \ref{thm:exact-lower}. Consider $r$ with $r(i)=i$. Then, for each $l\in[n]$, we choose an $r_l\in\mathcal{R}$ that differs with $r$ only at the $j$th position and $\|r_l-r\|^2=1$. By the same argument used in the proof of Theorem \ref{thm:exact-lower}, the lower bound is determined by the following probability,
$$\mathbb{P}_r\left(\max_{l\in[n]}\prod_{1\leq i\neq j\leq n}p(X_{ij}|\mu_{r_l(i)r_l(j)})> \prod_{1\leq i\neq j\leq n}p(X_{ij}|\mu_{r(i)r(j)})\right),$$
where we use $p(X|\mu)$ to denote the probability mass function of $\text{Poisson}(\mu)$.
By the construction of $r_l$'s, the probability above equals
\begin{equation}
\mathbb{P}_r\left(\max_{l\in[n]}\prod_{i\in[n]\backslash\{l\}}\frac{p(X_{il}|\mu_{r_l(i)r_l(l)})}{p(X_{il}|\mu_{r(i)r(l)})}\frac{p(X_{li}|\mu_{r_l(l)r_l(i)})}{p(X_{li}|\mu_{r(l)r(i)})}>1\right).\label{eq:enemy}
\end{equation}
Since $\max_{i,l}\left|\frac{\mu_{r_l(i)r_l(l)}-\mu_{r(i)r(l)}}{\mu_{r(i)r(l)}}\right|=o(1)$, a Taylor expansion argument leads to
\begin{eqnarray*}
&& \sum_{i\in[n]\backslash\{l\}}\left(\mu_{r(i)r(l)}\log\frac{\mu_{r(i)r(l)}}{\mu_{r_l(i)r_l(l)}}+\mu_{r_l(i)r_l(l)}-\mu_{r(i)r(l)}\right) \\
&=& 2(1+o(1))\sum_{i\in[n]\backslash\{l\}}\left(\sqrt{\mu_{r_l(i)r_l(l)}}-\sqrt{\mu_{r(i)r(l)}}\right)^2 \\
&=& 2n(1+o(1))\beta^2.
\end{eqnarray*}
Therefore,
$$\log \prod_{i\in[n]\backslash\{l\}}\frac{p(X_{il}|\mu_{r_l(i)r_l(l)})}{p(X_{il}|\mu_{r(i)r(l)})}=\sum_{i\in[n]\backslash\{l\}}(X_{il}-\mu_{r(i)r(l)})\log\frac{\mu_{r_l(i)r_l(l)}}{\mu_{r(i)r(l)}}-2n(1+o(1))\beta^2,$$
and thus we can lower bound (\ref{eq:enemy}) by
$$\mathbb{P}_r\left(\max_{l\in[n]}Y_l>\sqrt{2n}(1+o(1))\beta\right),$$
where
$$Y_l=\frac{1}{\sqrt{2(n-1)}}\left(\sum_{i\in[n]\backslash\{l\}}\frac{X_{il}-\mu_{r(i)r(l)}}{2\beta}\log\frac{\mu_{r_l(i)r_l(l)}}{\mu_{r(i)r(l)}}+\sum_{i\in[n]\backslash\{l\}}\frac{X_{li}-\mu_{r(l)r(i)}}{2\beta}\log\frac{\mu_{r_l(l)r_l(i)}}{\mu_{r(l)r(i)}}\right).$$
Note that
$$
\Var\left(\frac{X_{il}-\mu_{r(i)r(l)}}{2\beta}\log\frac{\mu_{r_l(i)r_l(l)}}{\mu_{r(i)r(l)}}\right) = \frac{\mu_{r(i)r(l)}}{4\beta^2}\left(\log\frac{\mu_{r_l(i)r_l(l)}}{\mu_{r(i)r(l)}}\right)^2 = 1+o(1),
$$
where the last equality is by a Taylor expansion argument. Moreover, the Poisson tail of $\frac{X_{il}-\mu_{r(i)r(l)}}{2\beta}\log\frac{\mu_{r_l(i)r_l(l)}}{\mu_{r(i)r(l)}}$ is well behaved. Therefore, we can apply a high-dimensional Gaussian approximation result by \cite{chernozhukov2013gaussian}, and obtain
$$\left|\mathbb{P}_r\left(\max_{l\in[n]}Y_l>\sqrt{2n}(1+o(1))\beta\right)-\mathbb{P}_r\left(\max_{l\in[n]}W_l>\sqrt{2n}(1+o(1))\beta\right)\right|=o(1),$$
where $W_1,...,W_n$ are jointly Gaussian with zero mean, and the covariance structure is determined by $\mathbb{E}W_l^2=o(1)$ and $\max_{j\leq l}\mathbb{E}W_jW_l=O(n^{-1})$. Therefore, by Lemma \ref{lem:hartigan} and the condition $\limsup_n\frac{n\beta^2}{\log n}<1$, we have
$$\mathbb{P}_r\left(\max_{l\in[n]}Y_l>\sqrt{2n}(1+o(1))\beta\right)\geq \frac{1}{2}-o(1),$$
which implies the desired result. The proof is complete.
\end{proof}

\subsection{Proofs of Auxiliary Lemmas}\label{sec:pf-aux}

\begin{proof}[Proof of Lemma \ref{lem:dev}]
We first bound $\mathbb{P}_r\left(L(\tilde{r})\leq L(r)\right)$ for any $\tilde{r},r\in\mathcal{R}$. Direct calculation gives
\begin{eqnarray*}
&& \mathbb{P}_r\left(L(\tilde{r})\leq L(r)\right) \\
&=& \mathbb{P}_r\left(2\sum_{1\leq i\neq j\le n}(\mu_{r(i)r(j)}-\mu_{\tilde{r}(i)\tilde{r}(j)})(X_{ij}-\mu_{r(i)r(j)})\leq -\sum_{1\leq i\neq j\leq n}(\mu_{\tilde{r}(i)\tilde{r}(j)}-\mu_{r(i)r(j)})^2\right).
\end{eqnarray*}
By the condition (\ref{eq:error}), the above probability is upper bounded by
$$\exp\left(-\frac{1}{8\sigma^2}\sum_{1\leq i\neq j\leq n}(\mu_{\tilde{r}(i)\tilde{r}(j)}-\mu_{r(i)r(j)})^2\right),$$
which is further bounded by $\exp\left(-\frac{n\beta^2}{4\sigma^2}\|\tilde{r}-r\|^2\right)$ according to the condition (\ref{eq:signal}). Therefore, we have
$$\sup_{\tilde{r}\in\mathcal{R}_m}\mathbb{P}_r\left(L(\tilde{r})\leq L(r)\right)\leq \sup_{\tilde{r}\in\mathcal{R}_m}\exp\left(-\frac{n\beta^2}{4\sigma^2}\|\tilde{r}-r\|^2\right)=\exp\left(-\frac{n\beta^2m}{4\sigma^2}\right).$$
The proof is complete.
\end{proof}

\begin{proof}[Proof of Lemma \ref{lem:card}]
To get an element in $\mathcal{R}_m$, we can first pick $l\in[m]$ and then only focused on those $\tilde{r}$ that differs with $r$ at $l$ positions. Next, we pick $l$ positions from $[n]$. The error at each position is defined as $n_1,...n_l$, respectively. They must satisfy $n_1^2+\cdots +n_l^2=m$. Finally, for each $i$ that belongs to the $l$ positions, $\tilde{r}(i)$ either takes $r(i)-n_i$ or $r(i)+n_i$. This gives the bound
$$|\mathcal{R}_m|\leq \sum_{l=1}^{\min(m,n)}{n\choose l}\sum_{n_1^2+\cdots+n_l^2=m}2^l.$$
We have
$$\sum_{n_1^2+\cdots+n_l^2=m}2^l\leq {m-1\choose l-1}2^l\leq\exp\left(l\log\frac{2em}{l}\right).$$
For $1\leq m\leq n$, we have
$$|\mathcal{R}_m|\leq (2e)^m\sum_{l=1}^m{n\choose l}\leq\exp\left(m\log(2e)+m\log\left(\frac{en}{m}\right)\right)=\left(\frac{2e^2n}{m}\right)^m.$$
For $n<m\leq n^2$, we have
$$|\mathcal{R}_m|\leq 2^n\sum_{l=1}^n\exp\left(l\log\frac{2em}{l}\right)\leq 2^n n \exp\left(n\log\frac{2em}{n}\right)\leq\left(\frac{8em}{n}\right)^n.$$
Finally, for $m>n^2$, we have $|\mathcal{R}_m|\leq |\mathcal{R}|\leq n^n$, and the proof is complete.
\end{proof}

\begin{proof}[Proof of Lemma \ref{lem:uniform}]
We first introduce some notation. For any $r\in\mathcal{R}$, we use $\mu(r)$ to denote an $n\times n$ matrix whose diagonal entries are all $0$ and off-diagonal entries are given by $\mu(r)_{ij}=\mu_{r(i)r(j)}$. Moreover, for any matrix that appears in this proof, its off-diagonal entries are all $0$. In this way, we can write $L(r)=\fnorm{Z}^2$ and $L(\tilde{r})=\fnorm{X-\mu(\tilde{r})}^2$. We give a lower bound for $L(\tilde{r})$,
\begin{eqnarray*}
L(\tilde{r}) &=& \fnorm{Z+\mu(r)-\mu(\tilde{r})}^2 \\
&=& \fnorm{Z}^2 + \fnorm{\mu(r)-\mu(\tilde{r})}^2 + 2\iprod{Z}{\mu(r)-\mu(\tilde{r})} \\
&\geq& \fnorm{Z}^2 + 2n\beta^2l^2t^2  + 2\iprod{Z}{\mu(r)-\mu(\tilde{r})},
\end{eqnarray*}
where the last inequality is by
$$\fnorm{\mu(r)-\mu(\tilde{r})}^2\geq 2n\beta^2\|\tilde{r}-r\|^2\geq 2n\beta^2l^2t^2,$$
according to the condition (\ref{eq:signal}) and $\|\tilde{r}-r\|\geq lt$. Therefore, we have
$$\min_{\{\tilde{r}\in\mathcal{R}: tl<\|\tilde{r}-r\|\leq t(l+1)\}} L(\tilde{r})- L(r)\geq 2n\beta^2l^2t^2  + 2\min_{\{\tilde{r}\in\mathcal{R}:\|\tilde{r}-r\|\leq t(l+1)\}}\iprod{Z}{\mu(r)-\mu(\tilde{r})},$$
and we obtain the bound
\begin{eqnarray*}
&& \mathbb{P}_{r}\left(\min_{\{\tilde{r}\in\mathcal{R}: tl<\|\tilde{r}-r\|\leq t(l+1)\}} L(\tilde{r})\leq L(r)\right) \\
&\leq& \mathbb{P}\left(\max_{\{\tilde{r}\in\mathcal{R}:\|\tilde{r}-r\|\leq t(l+1)\}}\left|\iprod{Z}{\mu(r)-\mu(\tilde{r})}\right|\geq n\beta^2l^2t^2\right) \\
&\lesssim& \exp\left(-C\frac{n^2\beta^4l^4t^4}{\left\|\max_{\{\tilde{r}\in\mathcal{R}:\|\tilde{r}-r\|\leq t(l+1)\}}\left|\iprod{Z}{\mu(r)-\mu(\tilde{r})}\right|\right\|^2_{\psi_2}}\right)
\end{eqnarray*}
The norm $\|\cdot\|_{\psi_2}$ is the Orlicz norm with function $\psi_2(x)=e^{x^2}-1$.

Now it is sufficient to bound $\left\|\max_{\{\tilde{r}\in\mathcal{R}:\|\tilde{r}-r\|\leq t(l+1)\}}\left|\iprod{Z}{\mu(r)-\mu(\tilde{r})}\right|\right\|^2_{\psi_2}$.
We let the truth $r$ be fixed. Then, for any $\tilde{r}\in\mathcal{R}$, define $Z(\tilde{r})=\iprod{Z}{\mu(r)-\mu(\tilde{r})}$, a sub-Gaussian process on $\mathcal{R}$. Note that for any $\tilde{r}_1,\tilde{r}_2$, we have
\begin{eqnarray}
\label{eq:GP-tail}\mathbb{P}\left(|Z(\tilde{r}_1)-Z(\tilde{r}_2)|>t\right) &\leq& 2\exp\left(-\frac{t^2}{2\sigma^2\fnorm{\mu(\tilde{r})-\mu(\tilde{r})}^2}\right) \\
\label{eq:GP-dist}&\leq& 2\exp\left(-\frac{t^2}{4Mn\sigma^2\beta^2\|\tilde{r}_1-\tilde{r}_2\|^2}\right),
\end{eqnarray}
where the inequality (\ref{eq:GP-tail}) is by (\ref{eq:error}), and the inequality (\ref{eq:GP-dist}) is by (\ref{eq:signal-entropy}). Therefore, by Lemma 2.2.1 of \cite{van1996weak}, the natural semimetric between $Z(\tilde{r}_1)$ and $Z(\tilde{r}_2)$ is
$$d(\tilde{r}_1,\tilde{r}_2)=\sqrt{12Mn\sigma^2\beta^2\|\tilde{r}_1-\tilde{r}_2\|^2}.$$
We use Corollary 2.2.5 of \cite{van1996weak}, and get
\begin{eqnarray*}
&& \left\|\max_{\{\tilde{r}\in\mathcal{R}:\|\tilde{r}-r\|\leq t(l+1)\}}\left|\iprod{Z}{\mu(r)-\mu(\tilde{r})}\right|\right\|_{\psi_2} \\
&\leq& \left\|\max_{\{\tilde{r}\in\mathcal{R}:d(\tilde{r},r)\leq \sqrt{12Mn\sigma^2\beta^2t^2(l+1)^2}\}}\left|\iprod{Z}{\mu(r)-\mu(\tilde{r})}\right|\right\|_{\psi_2} \\
&\lesssim& \int_0^{\sqrt{12  Mn\sigma^2\beta^2\|r-r_0\|^2}} \sqrt{\log\left(1+N(\epsilon,\bar{\mathcal{R}}_l,d)\right)}d\epsilon,
\end{eqnarray*}
where $N(\epsilon,\bar{\mathcal{R}}_l,d)$ denotes the covering number of
$$\bar{\mathcal{R}}_l=\{\tilde{r}\in\mathcal{R}:d(\tilde{r},r)\leq \sqrt{12Mn\sigma^2\beta^2t^2(l+1)^2}\}$$
 with respect to the distance $d$ and radius $\epsilon$. A standard volume ratio argument gives
$$\int_0^{\sqrt{12  Mn\sigma^2\beta^2t^2(l+1)^2}} \sqrt{\log\left(1+N(\epsilon,\bar{\mathcal{R}}_l,d)\right)}d\epsilon\lesssim \sqrt{  M n^2\sigma^2\beta^2t^2(l+1)^2}.$$
Finally, the exponent has lower bound
$$\frac{n^2\beta^4l^4t^4}{\left\|\max_{\{\tilde{r}\in\mathcal{R}:\|\tilde{r}-r\|\leq t(l+1)\}}\left|\iprod{Z}{\mu(r)-\mu(\tilde{r})}\right|\right\|^2_{\psi_2}}\gtrsim \frac{\beta^2t^2l^2}{M\sigma^2},$$
which completes the proof.
\end{proof}

\begin{proof}[Proof of Lemma \ref{lem:dev-comp}]
We use the notation
$\Delta_r=\sum_{j=1}^n\theta_{r(j)}-\sum_{j=1}^n\theta_j$.
By (\ref{eq:stable-comp}), we have $\max_{r\in\mathcal{R}}|\Delta_r|=o(\sqrt{n}\beta)$. Note that for each $i$, $\mathbb{E}_rS_i=\theta_{r(i)}-\Delta_r/n$. With some direct calculation, the event $\sum_{i=1}^n(S_i-\theta_{\tilde{r}(i)})^2\leq\sum_{i=1}^n(S_i-\theta_{r(i)})^2$ is equivalent to
$$2\sum_{i=1}^n(\theta_{r(i)}-\theta_{\tilde{r}(i)})(S_i-\theta_{r(i)}+\Delta_r/n)\leq -\sum_{i=1}^n(\theta_{r(i)}-\theta_{\tilde{r}(i)})^2+2\Delta_r(\Delta_r-\Delta_{\tilde{r}})/n.$$
Using the conditions (\ref{eq:signal-comp}) and (\ref{eq:stable-comp}), we get the bound
$$-\sum_{i=1}^n(\theta_{r(i)}-\theta_{\tilde{r}(i)})^2+2\Delta_r(\Delta_r-\Delta_{\tilde{r}})/n\leq -(1+o(1))\sum_{i=1}^n(\theta_{r(i)}-\theta_{\tilde{r}(i)})^2.$$
Therefore, it suffices to bound
$$\mathbb{P}_r\left(2\sum_{i=1}^n(\theta_{r(i)}-\theta_{\tilde{r}(i)})(S_i-\theta_{r(i)}+\Delta_r/n)\leq -(1+o(1))\sum_{i=1}^n(\theta_{r(i)}-\theta_{\tilde{r}(i)})^2\right).$$
By (\ref{eq:error}), we get the bound
$$\exp\left(-(1+o(1))\frac{n\sum_{i=1}^n(\theta_{r(i)}-\theta_{\tilde{r}(i)})^2}{4\sigma^2}\right).$$
Finally, using (\ref{eq:signal-comp}), we obtain the desired rate, and the proof is complete.
\end{proof}

\begin{proof}[Proof of Lemma \ref{lem:dev-add}]
Note that for the $S_i$ defined in (\ref{eq:score-coll}), we have $\mathbb{E}_rS_i=\theta_{r(i)}$. Some direct calculation gives
\begin{eqnarray*}
&& \mathbb{P}_r\left(\sum_{i=1}^n(S_i-\theta_{\tilde{r}(i)})^2\leq\sum_{i=1}^n(S_i-\theta_{r(i)})^2\right) \\
&\leq& \mathbb{P}_r\left(2\sum_{i=1}^n(\theta_{r(i)}-\theta_{\tilde{r}(i)})(S_i-\theta_{r(i)})\leq -\sum_{i=1}^n(\theta_{r(i)}-\theta_{\tilde{r}(i)})^2\right).
\end{eqnarray*}
Using (\ref{eq:error}) and (\ref{eq:signal-comp}), we obtain the bound
$$\exp\left(-(1+o(1))\frac{n\sum_{i=1}^n(\theta_{r(i)}-\theta_{\tilde{r}(i)})^2}{4\sigma^2}\right)\leq \exp\left(-(1+o(1))\frac{n\beta^2\|\tilde{r}-r\|^2}{4\sigma^2}\right),$$
and the proof is complete.
\end{proof}

\begin{proof}[Proof of Lemma \ref{lem:dev-PL}]
We use $\mathbbm{1}\in\mathbb{R}^n$ to denote a column vector with entries all $1$. Given an $r\in\mathcal{R}'$, define the hat matrix by
\begin{equation}
H_r=\frac{1}{n}\mathbbm{1}\mathbbm{1}^T + \frac{\left(I-\frac{1}{n}\mathbbm{1}\mathbbm{1}^T\right)rr^T\left(I-\frac{1}{n}\mathbbm{1}\mathbbm{1}^T\right)}{\left\|\left(I-\frac{1}{n}\mathbbm{1}\mathbbm{1}^T\right)r\right\|^2}.\label{eq:hat-matrix}
\end{equation}
Using basic facts of linear regression, we get
$$\text{PL}(r)=\|(I-H_r)\hat{S}\|^2.$$
Then, with some rearrangement, the inequality $\text{PL}(\tilde{r})\leq \text{PL}(r)$ is equivalent to
\begin{eqnarray}
\nonumber && 2\iprod{(I-H_{\tilde{r}})\mathbb{E}_r\hat{S}}{\hat{S}-\mathbb{E}_r\hat{S}} + \|H_r(\hat{S}-\mathbb{E}_r\hat{S})\|^2 - \|H_{\tilde{r}}(\hat{S}-\mathbb{E}_r\hat{S})\|^2  \\
\label{eq:pull} &\leq& -\left\|(I-H_{\tilde{r}})\mathbb{E}_r\hat{S}\right\|^2
\end{eqnarray}
Note that
$$\mathbb{E}_r\hat{S}=\tilde{\beta}r + C\mathbbm{1},$$
where the value of $C$ depends on whether it is the comparison model or the collaboration model. According to the definition of the projection matrix $I-H_r$, we have $(I-H_r)\mathbbm{1}=0$, so that $(I-H_{\tilde{r}})\mathbb{E}_r\hat{S}=\tilde{\beta}(I-H_{\tilde{r}})r$.

Now we study $\left\|(I-H_{\tilde{r}})r\right\|^2$. Define
$$
x=\left(I-\frac{1}{n}\mathbbm{1}\mathbbm{1}^T\right)r\quad\text{and}\quad y=\left(I-\frac{1}{n}\mathbbm{1}\mathbbm{1}^T\right)\tilde{r}.
$$
Then,
\begin{equation}
\left\|(I-H_{\tilde{r}})r\right\|^2=\left\|x-\frac{y^Tx}{\|y\|^2}y\right\|^2=\|x-y\|^2-\frac{|(x-y)^Ty|^2}{\|y\|^2}.\label{eq:lunar}
\end{equation}
Since
$$\left\|\frac{1}{n}\mathbbm{1}\mathbbm{1}^Tr-\frac{1}{n}\mathbbm{1}\mathbbm{1}^T\tilde{r}\right\|^2=\frac{1}{n}\left(\sum_{i=1}^nr(i)-\sum_{i=1}^n\tilde{r}(i)\right)^2=O(c_n^2/n)=o(1),$$
we have
$$\|x-y\|^2=\left\|(r-\tilde{r})-\left(\frac{1}{n}\mathbbm{1}\mathbbm{1}^Tr-\frac{1}{n}\mathbbm{1}\mathbbm{1}^T\tilde{r}\right)\right\|^2=(1+o(1))\|r-\tilde{r}\|^2.$$
Moreover,
\begin{eqnarray}
\nonumber \frac{|(x-y)^Ty|^2}{\|y\|^2} &\leq& \frac{2\left|(x-y)^T\left(\frac{x-y}{2}\right)\right|^2}{\|y\|^2} + \frac{2\left|(x+y)^T\left(\frac{x-y}{2}\right)\right|^2}{\|y\|^2} \\
\label{eq:spirit} &=& \frac{\|x-y\|^2}{2\|y\|^2} + \frac{|\|x\|^2-\|y\|^2|^2}{2\|y\|^2}.
\end{eqnarray}
Note that $\|y\|^2\gtrsim n^3$. Thus, the first term of (\ref{eq:spirit}) is bounded by $o(1)\|\tilde{r}-r\|^2$. The second term of (\ref{eq:spirit}) is bounded by $o(1)$ according to the definition of $\mathcal{R}'$. Therefore,
\begin{equation}
\left\|(I-H_{\tilde{r}})r\right\|^2\geq (1+o(1))\|\tilde{r}-r\|^2.\label{eq:under}
\end{equation}

We have
\begin{eqnarray*}
\mathbb{P}_r\left(\text{PL}(\tilde{r})\leq \text{PL}(r)\right) &\leq& \mathbb{P}_r\left(2\tilde{\beta}\iprod{(I-H_{\tilde{r}})r}{\hat{S}-\mathbb{E}_r\hat{S}}\leq -(1-t)\tilde{\beta}^2\|(I-H_{\tilde{r}})r\|^2\right) \\
&& + \mathbb{P}_r\left( \|H_r(\hat{S}-\mathbb{E}_r\hat{S})\|^2 - \|H_{\tilde{r}}(\hat{S}-\mathbb{E}_r\hat{S})\|^2\leq -t\tilde{\beta}^2\|(I-H_{\tilde{r}})r\|^2\right),
\end{eqnarray*}
and we will bound the two terms separately. Using (\ref{eq:error}), we can bound the first term by
$$\exp\left(-(1+o(1))\frac{(1-t)n\beta^2\|(I-H_{\tilde{r}})r\|^2}{4\sigma^2}\right)\leq \exp\left(-(1+o(1))\frac{(1-t)n\beta^2\|\tilde{r}-r\|^2}{4\sigma^2}\right).$$
For the second term, we write $A=(H_{\tilde{r}}-H_r)(H_{\tilde{r}}+H_r)$, and then
\begin{eqnarray*}
&& \mathbb{P}_r\left( \|H_r(\hat{S}-\mathbb{E}_r\hat{S})\|^2 - \|H_{\tilde{r}}(\hat{S}-\mathbb{E}_r\hat{S})\|^2\leq -t\tilde{\beta}^2\|(I-H_{\tilde{r}})r\|^2\right) \\
&\leq& \mathbb{P}_r\left(Z^TAZ \geq t\frac{2n\beta^2}{\sigma^2}\|\tilde{r}-r\|^2\right),
\end{eqnarray*}
where $Z=\sqrt{2n}(\hat{S}-\mathbb{E}_r\hat{S})/\sigma$. Suppose the eigenvalue decomposition of $A$ is $\sum_l d_lu_lu_l^T$, we define $\bar{A}=\sum_l \max(0,d_l)u_lu_l^T$, and it is sufficient to bound $\mathbb{P}_r\left(Z^T\bar{A}Z \geq t\frac{2n\beta^2}{\sigma^2}\|\tilde{r}-r\|^2\right)$. By Hanson-Wright inequality \citep{hsu2012tail}, we have
\begin{eqnarray*}
&& \mathbb{P}_r\left(Z^T\bar{A}Z \geq t\frac{2n\beta^2}{\sigma^2}\|\tilde{r}-r\|^2\right) \\
&\leq& \exp\left(-C\min\left\{\frac{\left(t\frac{2n\beta^2}{\sigma^2}\|\tilde{r}-r\|^2-\Tr(\bar{A})\right)^2}{\fnorm{A}^2},\frac{t\frac{2n\beta^2}{\sigma^2}\|\tilde{r}-r\|^2-\Tr(\bar{A})}{\opnorm{A}}\right\}\right).
\end{eqnarray*}
Since $\rank{(A)}\leq 2$, we have $\Tr(\bar{A})\leq 2\opnorm{A}\leq 2\fnorm{A}\leq 4\fnorm{H_r-H_{\tilde{r}}}\leq C_1\sqrt{\frac{\|\tilde{r}-r\|^2}{n^3}}$. Under the condition $\frac{n\beta^2}{4\sigma^2}>1$, we have
$$\mathbb{P}_r\left(Z^T\bar{A}Z \geq t\frac{2n\beta^2}{\sigma^2}\|\tilde{r}-r\|^2\right)\leq \exp\left(-C_2tn^{3/2}\|\tilde{r}-r\|\frac{n\beta^2}{\sigma^2}\right).$$

Combining the above two bounds, we obtain
$$\mathbb{P}_r\left(\text{PL}(\tilde{r})\leq \text{PL}(r)\right)\leq \exp\left(-(1+o(1))\frac{(1-t)n\beta^2\|\tilde{r}-r\|^2}{4\sigma^2}\right) + \exp\left(-C_2tn^{3/2}\|\tilde{r}-r\|\frac{n\beta^2}{\sigma^2}\right).$$
For $\tilde{r}\in\mathcal{R}_m$, we choose $t=\frac{1}{4C_2}\sqrt{\frac{m}{n^3}}$, and we obtain
$$\mathbb{P}_r\left(\text{PL}(\tilde{r})\leq \text{PL}(r)\right)\leq 2\exp\left(-(1+o(1))\left(1-\frac{1}{4C_2}\sqrt{\frac{m}{n^3}}\right)\frac{mn\beta^2}{4\sigma^2}\right).$$
The proof is complete.
\end{proof}

\begin{proof}[Proof of Lemma \ref{lem:uniform-PL}]
We will borrow arguments and notations from the proof of Lemma \ref{lem:dev-PL}. Recall the hat matrix defined in (\ref{eq:hat-matrix}). By (\ref{eq:lunar}), we have
$$\left\|(I-H_{\tilde{r}})r\right\|^2\leq \|r-\tilde{r}\|^2.$$
Then using (\ref{eq:under}), we obtain
$$\left\|(I-H_{\tilde{r}})r\right\|^2=(1+o(1)) \|r-\tilde{r}\|^2.$$
Using the same argument that derives (\ref{eq:pull}), we have
\begin{eqnarray*}
&& \min_{\{\tilde{r}\in\mathcal{R}': tl<\|\tilde{r}-r\|\leq t(l+1)\}} \text{PL}(\tilde{r})-\text{PL}(r) \\
&\geq&\min_{\{\tilde{r}\in\mathcal{R}': tl<\|\tilde{r}-r\|\leq t(l+1)\}}\left( 2\tilde{\beta}\iprod{(I-H_{\tilde{r}})r}{\hat{S}-\mathbb{E}_r\hat{S}} +\tilde{\beta}^2\left\|(I-H_{\tilde{r}})r\right\|^2/2\right) \\
&& + \min_{\{\tilde{r}\in\mathcal{R}': tl<\|\tilde{r}-r\|\leq t(l+1)\}}\left(\|H_r(\hat{S}-\mathbb{E}_r\hat{S})\|^2 - \|H_{\tilde{r}}(\hat{S}-\mathbb{E}_r\hat{S})\|^2+\tilde{\beta}^2\left\|(I-H_{\tilde{r}})r\right\|^2/2\right) \\
&\geq& \min_{\{\tilde{r}\in\mathcal{R}': \|\tilde{r}-r\|\leq t(l+1)\}}2\tilde{\beta}\iprod{(I-H_{\tilde{r}})r}{\hat{S}-\mathbb{E}_r\hat{S}} + (1+o(1))\tilde{\beta}^2t^2l^2/2 \\
&& + \min_{\{\tilde{r}\in\mathcal{R}': tl<\|\tilde{r}-r\|\leq t(l+1)\}}\left(\|H_r(\hat{S}-\mathbb{E}_r\hat{S})\|^2 - \|H_{\tilde{r}}(\hat{S}-\mathbb{E}_r\hat{S})\|^2\right)  + (1+o(1))\tilde{\beta}^2t^2l^2/2 \\
&\geq& \min_{\{v\in\mathbb{R}^n: \|v\|\leq 2t(l+1)\}}2\tilde{\beta}\iprod{v}{\hat{S}-\mathbb{E}_r\hat{S}} + (1+o(1))\tilde{\beta}^2t^2l^2/2 \\
&& + \min_{\{A=A^T\in\mathbb{R}^{n\times n}:\fnorm{A}\leq Ct(l+1)n^{-3/2},\rank(A)\leq 2\}} (\hat{S}-\mathbb{E}_r\hat{S})^TA(\hat{S}-\mathbb{E}_r\hat{S}) +(1+o(1))\tilde{\beta}^2t^2l^2/2.
\end{eqnarray*}
Therefore,
\begin{eqnarray}
\nonumber && \mathbb{P}_{r}\left(\min_{\{\tilde{r}\in\mathcal{R}': tl<\|\tilde{r}-r\|\leq t(l+1)\}} \text{PL}(\tilde{r})\leq \text{PL}(r)\right) \\
\nonumber &\leq& \mathbb{P}_r\left(\max_{\{v\in\mathbb{R}^n: \|v\|\leq 2t(l+1)\}}\left|\iprod{v}{\hat{S}-\mathbb{E}_r\hat{S}}\right|\geq |\tilde{\beta}|t^2l^2/8\right) \\
\nonumber && + \mathbb{P}_r\left(\max_{\{A=A^T\in\mathbb{R}^{n\times n}:\fnorm{A}\leq Ct(l+1)n^{-3/2},\rank(A)\leq 2\}} (\hat{S}-\mathbb{E}_r\hat{S})^TA(\hat{S}-\mathbb{E}_r\hat{S})\geq |\tilde{\beta}|^2t^2l^2/4\right) \\
\label{eq:metropolis} &\leq& \mathbb{P}_r\left(\max_{\{v\in\mathbb{R}^n: \|v\|\leq 1\}}\left|\iprod{v}{Z}\right|\geq \frac{\sqrt{n}|\tilde{\beta}|tl}{32\sigma}\right) \\
\label{eq:metropolis-pt2} && + \mathbb{P}_r\left(\max_{\{A=A^T\in\mathbb{R}^{n\times n}:\fnorm{A}\leq 1,\rank(A)\leq 2\}} Z^TAZ\geq \frac{n^{5/2}\beta^2tl}{4C\sigma^2}\right),
\end{eqnarray}
where $Z=\sqrt{2n}(\hat{S}-\mathbb{E}_r\hat{S})/\sigma$.

A standard discretization argument in \cite[Lemma A.1]{gao2015rate} gives
$$\max_{\{v\in\mathbb{R}^n: \|v\|\leq 1\}}\left|\iprod{v}{Z}\right|\leq 2\max_{1\leq j\leq J}\left|\iprod{v_j}{Z}\right|,$$
where $\{v_j\}_{1\leq j\leq J}$ is a subset of $\{v\in\mathbb{R}^n: \|v\|\leq 1\}$ such that for any $v\in\mathbb{R}^n$ with $\|v\|\leq 1$, there is some $j\in[J]$ that satisfies $\|v_j-v\|\leq 1/2$, and a covering number argument gives the bound $J\leq e^{5n}$. By union bound, we can bound (\ref{eq:metropolis}) by
\begin{eqnarray*}
&& \mathbb{P}_r\left(2\max_{1\leq j\leq J}\left|\iprod{v_j}{Z}\right|\geq \frac{\sqrt{n}|\tilde{\beta}|tl}{32\sigma}\right)\\
&\leq& \sum_{j=1}^J\mathbb{P}_r\left(2\left|\iprod{v_j}{Z}\right|\geq \frac{\sqrt{n}|\tilde{\beta}|tl}{32\sigma}\right)\\
&\leq& 2\exp\left(5n - C'\frac{n\beta^2t^2l^2}{\sigma^2}\right),
\end{eqnarray*}
for some constant $C'>0$, and the last inequality above is by (\ref{eq:error}). 

For (\ref{eq:metropolis-pt2}), a similar discretization argument gives
$$\max_{A\in\mathcal{A}}Z^TAZ \leq 8\max_{1\leq j\leq L}Z^TA_jZ,$$
where $\mathcal{A}=\{A=A^T\in\mathbb{R}^{n\times n}:\fnorm{A}\leq 1,\rank(A)\leq 2\}$, and $\}A_j\}_{j\in[L]}\subset\mathcal{A}$ with $L$ satisfying $L\leq e^{C_1n}$ for some constant $C_1>0$. Then, we can bound (\ref{eq:metropolis-pt2}) by
\begin{eqnarray*}
&& \mathbb{P}_r\left(8\max_{1\leq j\leq L}Z^TA_jZ\geq \frac{n^{5/2}\beta^2tl}{4C\sigma^2}\right) \\
&\leq& \sum_{j=1}^L\mathbb{P}_r\left(8Z^TA_jZ\geq \frac{n^{5/2}\beta^2tl}{4C\sigma^2}\right) \\
&\leq& \exp\left(C_1n - C_2\frac{n^{5/2}\beta^2 tl}{\sigma^2}\right),
\end{eqnarray*}
where the last inequality above is by Hanson-Wright inequality \citep{hsu2012tail} and (\ref{eq:error}). Since $\|\tilde{r}-r\|$ is at most $n^{3/2}$, we only need to consider $tl\lesssim n^{3/2}$, which implies
$$\exp\left(C_1n - C_2\frac{n^{5/2}\beta^2 tl}{\sigma^2}\right)\leq \exp\left(C_1n - C_2'\frac{n\beta^2 t^2l^2}{\sigma^2}\right).$$
Combining the bounds for (\ref{eq:metropolis}) and (\ref{eq:metropolis-pt2}), we obtain the desired conclusion.
\end{proof}

\begin{proof}[Proof of Lemma \ref{lem:dev-Poi}]
For any $\tilde{r}\in\mathcal{R}_m$, we have
\begin{eqnarray*}
&& \mathbb{P}_r\left(\prod_{1\leq i\neq j\leq n}\frac{\mu_{\tilde{r}(i)\tilde{r}(j)}^{X_{ij}}e^{-\mu_{\tilde{r}(i)\tilde{r}(j)}}}{X_{ij}!}\geq \prod_{1\leq i\neq j\leq n}\frac{\mu_{r(i)r(j)}^{X_{ij}}e^{-\mu_{r(i)r(j)}}}{X_{ij}!}\right) \\
&=& \mathbb{P}_r\left(\prod_{1\leq i\neq j\leq n} \left(\frac{\mu_{\tilde{r}(i)\tilde{r}(j)}}{\mu_{r(i)r(j)}}\right)^{X_{ij}}e^{-\mu_{\tilde{r}(i)\tilde{r}(j)}+\mu_{r(i)r(j)}}\geq 1\right) \\
&\leq& \prod_{1\leq i\neq j\leq n}\mathbb{E}_r\sqrt{\left(\frac{\mu_{\tilde{r}(i)\tilde{r}(j)}}{\mu_{r(i)r(j)}}\right)^{X_{ij}}e^{-\mu_{\tilde{r}(i)\tilde{r}(j)}+\mu_{r(i)r(j)}}} \\
&=& \exp\left(-\frac{1}{2}\sum_{1\leq i\neq j\leq n}\left(\sqrt{\mu_{\tilde{r}(i)\tilde{r}(j)}}-\sqrt{\mu_{r(i)r(j)}}\right)^2\right).
\end{eqnarray*}
By (\ref{eq:signal-Poi}), we have obtain the upper bound $\exp\left(-n\beta^2\|\tilde{r}-r\|^2\right)$. The proof is complete by taking maximum over $\tilde{r}\in\mathcal{R}_m$.
\end{proof}

\section*{Acknowledgements}

The author thanks Qiyang Han for suggesting the proof of Theorem \ref{thm:upper2}. The research is supported in part by NSF grant DMS-1712957.

\bibliographystyle{plainnat}
\bibliography{rank}


\end{document}